\documentclass[a4paper,11pt,english]{amsart}
\usepackage{babel,amsthm,amsmath,amsfonts,amssymb,amscd,
stmaryrd,yhmath}
\usepackage{amssymb,xspace,amscd,txfonts,esvect, 
yfonts,calrsfs}
\usepackage{mathrsfs}
\DeclareMathAlphabet{\mathcalligra}{T1}{calligra}{m}{n}
\DeclareMathAlphabet{\mathpzc}{OT1}{pzc}{m}{it}
\usepackage{graphicx}
\usepackage{hyperref}
\input xy
\xyoption{all}
\CompileMatrices

\numberwithin{equation}{section}

\theoremstyle{plain}

\newtheorem{thm}{Theorem}[section]

\newtheorem{lem}[thm]{Lemma}

\newtheorem{cor}[thm]{Corollary}

\newtheorem{prop}[thm]{Proposition}

\theoremstyle{definition}

\newtheorem{defn}{Definition}[section]
\newtheorem{exam}[thm]{Example}

\newtheorem{ntz}{Notation}[section]

\newtheorem{rmk}[thm]{Remark}

\DeclareMathOperator{\Sub}{\mathpzc{S}}

\DeclareMathOperator{\cD}{\mathcal{D}}
\DeclareMathOperator{\cE}{\mathcal{E}}
\DeclareMathOperator{\cDd}{{\mathcal{D}^{\prime}}}
\DeclareMathOperator{\cEd}{{\mathcal{E}^{\prime}}}
\DeclareMathOperator{\cO}{\mathcal{O}}

\DeclareMathOperator{\R}{\mathbb{R}}

\DeclareMathOperator{\CiM}{\mathcal{C}_M^\infty}

\newcommand{\dc}{{d}^c_{\lambdaup}\!}
\newcommand{\df}{d\!}
\DeclareMathOperator{\Li}{\mathfrak{L}}

\DeclareMathOperator{\supp}{\mathrm{supp}}

\DeclareMathOperator{\pt}{\mathfrak{p}}
\DeclareMathOperator{\Pt}{P_{\!\tauup}}
\DeclareMathOperator{\Ptj}{P_{\!\tauup_j}}

\DeclareMathOperator{\im}{{\mathrm{Im}}}
\DeclareMathOperator{\re}{{\mathrm{Re}}}

\DeclareMathOperator{\Ot}{{\mathscr{O}}}

\DeclareMathOperator{\C}{\mathbb{C}}

\DeclareMathOperator\Co{\mathcal{C}}
\DeclareMathOperator\Ci{\mathcal{C}^\infty}
\DeclareMathOperator\Cic{\mathcal{C}_0^\infty}

\newcommand\Zf{\mathpzc{Z}}

\newcommand\Lf{\mathpzc{L}}

\newcommand\Vf{\mathpzc{V}}

\newcommand\Ub{\mathbf{U}}
\newcommand\GL{\mathbf{GL}}

\newcommand\If{\mathcal{I}\!}
\newcommand\Ifc{\bar{\If\,\,\,\,}_{\!\!\!\!\!\!{1}}}

\newcommand\Jf{\mathcal{J}\!}
\newcommand\Jfc{\mathcal{J}^{\mathbb{C}}_{\!\!{1}}}
\newcommand\Af{\mathcal{A}\!}
\newcommand\Afc{\mathcal{A}^{\mathbb{C}}_1\!\!}

\newcommand\Pf{\mathpzc{P}}
\newcommand{\Wf}{\mathpzc{W}}
\newcommand\Hess{\mathrm{Hess}^{1,1}}

\newcommand\Yf{\mathpzc{Y}}
\newcommand\Qf{\mathpzc{Q}}
\newcommand\SU{\mathbf{SU}}

\newcommand\St{\mathpzc{St}}
\newcommand\Gr{\mathpzc{Gr}}

\newcommand\Kf{[\ker\!\Lf]}
\newcommand\Hf{\mathpzc{H}}
\newcommand\Qc{\mathcal{Q}}
\newcommand\Sob{\mathrm{W}}
\newcommand\Ff{\mathpzc{F}}
\newcommand\Xf{\mathrm{X}}
\newcommand\Qt{Q_{\tauup}}

\title{Weak {$q$}-concavity
conditions
for {$CR$}-manifolds}

\author[M.~Nacinovich]{Mauro Nacinovich}
\address{M.\ Nacinovich:
Dipartimento di Matematica\\ II Universit\`a di Roma
``Tor Ver\-ga\-ta''\\ Via della Ricerca Scientifica\\ 00133 Roma
(Italy)}
\email{nacinovi@mat.uniroma2.it}
\author[E.~Porten]{Egmont Porten}
\address{E.\ Porten: Department of Mathematics\\ 
Mid Sweden University\\ 85170 Sundsvall \\Sweden\\
and \\
Instytut Matematyki\\
Uniwersytet Jana Kochanowskiego w Kielcach\\
Poland}
\email{Egmont.Porten@miun.se}

\date\today


\begin{document}
\begin{abstract}
We introduce various notions of q-pseudo-concavity for abstract CR manifolds and we apply these notions to the study of hyoo-ellipticity, maximum modulus principle and Cauchy problems for CR functions.
\end{abstract}
\maketitle
\tableofcontents

\section*{Introduction}
The definition of $q$-pseudo-concavity  for abstract $CR$ manifolds of arbitrary
$CR$-dimension and $CR$-codimension, given in  \cite{HN96}, required that 
all scalar Levi forms corresponding to non-characteristic codirections
have  Witt index\footnote{The Witt index of
a Hermitian form of signature $(p,q)$ is $\min\{p,q\}.$}
larger or equal to $q$. Important classes of
homogeneous examples 
(see e.g. \cite{AMN06, AMN2013,   AMN10x, AMN10,  MN00, MN97}) show that these conditions are
in fact too restrictive and that weaker notions of $q$-pseudo-concavity are needed. 
For example, the results on the non validity of the Poincar\'e lemma for
the tangential Cauchy-Riemann complex 
in 
\cite{BHN15, HN06} 
only involve scalar Levi forms of maximal rank.
In \cite{HN00} the
classical notion of 
$1$-pseudo-concavity was extended by a trace condition, that was 
further improved in \cite{AHNP08a, 33, HN03}. These notions are relevant to the behavior of $CR$
functions, being related to hypo-ellipticity, weak and strong unique continuation, 
hypo-analicity (see \cite{NaPo}) 
and 
the maximum modulus principle. \par
In this paper, we continue these investigations. A key point of this approach is the
simple observation that
the Hermitian symmetric
vector valued Levi form $\Lf$ 
of a $CR$ manifold $M$  defines a linear form on 
$T^{1,1}M=T^{1,0}M\otimes_MT^{0,1}M$. Our notion of
pseudoconcavity is the request that its kernel contains elements $\tauup$ which
are positive semidefinite. To  such a
 $\tauup$  
we can associate an invariantly defined 
degenerate-elliptic real partial differential operators $\Pt$, which 
turns out to be related to the $dd^c$-operator
of \cite{43}. By consistently keeping this perspective, we prove 
in this paper some results on $\Ci$-hypoellipticity, 
the maximum modulus principle, and undertake the study of boundary value problems for
$CR$ functions on open domains of abstract $CR$ manifolds, testing the effectiveness of a
new notion of weak 
two-pseudo-concavity by its application to the Cauchy problem for $CR$ functions. \par
The general plan of the paper is the following. In the first section we define the notion of $\Zf$-structure,
that generalizes $CR$ structures insofar that all formal integrability and rank conditions can be
dropped while our focus are $CR$ functions, only considered as  solutions of a homogeneous
overdetermined 
system of first order p.d.e.'s, and set the basic notation that will be used throughout the  paper.
In particular, we introduce the kernel $\Kf$ of the Levi form as a subsheaf of the sheaf of germs
of semipositive tensors of type $(1,1)$. 
\par In \S\ref{sec2} we show how the maximum modulus principle relates
to $\Ci$-regularity 
and
weak and strong unique continuation of $CR$ functions. 
We also make some comments on generic points of non-embeddable $CR$ manifolds, 
where, by using the results of \cite{NaPo},
we can prove,
in Proposition~\ref{prop2.5}, a result of strong unique continuation and partial hypo-analiticity
(cf. \cite{t92}).\par
In \S\ref{s3} we show how that  to each semi-positive tensor
$\tauup$ in the kernel of the Levi form we can associate
a real degenerate
elliptic scalar p.d.o. of the second order~$\Pt$. 
Real parts of $CR$ functions are $\Pt$-harmonic and the modulus of a $CR$ function is
$\Pt$-subharmonic at points where it is
different from zero.  Then, by using some techniques originally developed for the generalized
Kolmogorov equation (cf. \cite{Hor67, Hor85, K71}) we are able to 
enlarge, in comparison with  \cite{AHNP08a}, the set of vector fields \textit{enthralled by $\Zf$}.  
Thus we can improve, by Theorem~\ref{thm3.3},
some hypo-ellipticity result of \cite{AHNP08a}, and, by Theorem~\ref{thm3.7}, a  
propagation result
of \cite{HN03}, for the case in which this hypo-ellipticity fails.
\par
In \S\ref{s4} we prove the $CR$ analogue of Malgrange's theorem on the vanishing of the top degree cohomology
under some subellipticity condition. Our result slightly generalizes previous results of \cite{BHN10, CL1999, CL2000},
also yielding a Hartogs-type theorem on abstract $CR$ manifolds, 
to recover a $CR$ function on a relatively compact domain from  boundary values
satisfying  some momentum condition (Proposition~\ref{propo4.3}). \par
In \S\ref{sechopf} we use  the $dd^c$-operator of \cite{43} to show that the operators
$\Pt$ are invariantly defined in terms of 
sections of $\Kf$ (Corollary~\ref{cor4.8}).
The Hopf Lemma for $\Pt$ is used to deduce 
pseudo-convexity properties of the boundary
of a domain where a $CR$ functions has a peak point 
(Proposition~\ref{lm5.2}). This leads to a notion
of convexity/concavity for points of the boundary 
of a domain (Definition~\ref{d4.4}). Most of these notions
can be formulated in terms of the scalar Levi forms 
associated to 
the covectors of 
a half-space of the characteristic bundle. \par 
Thus in \S\ref{cv} we have found it convenient to consider properties of
convex cones of Hermitian symmetric forms satisfying conditions on their 
indices of inertia,
which are preliminary to 
the definitions of the next section. 
\par 
In \S\ref{s7} we propose various notions of weak-$q$-pseudoconcavity, give some examples, and show
in Proposition~\ref{prop7.7} that on an essentially-$2$-pseudo-concave manifold strong-$1$-convexity/concavity
at the boundary becomes an \textit{open} condition, i.e. stable under small perturbations. 
This is used in the last two sections  to discuss existence and uniqueness for the
Cauchy problem for $CR$ functions, with initial data on a hypersurface. \par
In \S\ref{s8}, 
after discussing uniqueness in 
the case of a locally embeddable $CR$ manifold, we turn to the case of an abstract
$CR$ manifold, proving, via Carleman-type estimates, that the uniqueness results of \cite{DCN, HN00, HN03}
can be extended by using some convexity condition (see Proposition~\ref{prop8.8}). 
In \S\ref{s9} an existence theorem for the Cauchy problem 
is proved for locally embeddable $CR$ manifolds, under some convexity conditions.

\section{\texorpdfstring{$CR$- and $\Zf$}{TEXT}-manifolds: preliminaries and notation} 
\label{s1}
Let $M$ be 
a real smooth  manifold of dimension $m$. 
\begin{defn}
A \emph{$\Zf$-structure}
on $M$ is the datum of a 
$\CiM$-submodule
$\Zf$ of the 
sheaf $\mathfrak{X}^{\C}_M$
of 
germs of smooth complex vector fields on $M$.
 It is called \begin{itemize}
 \item \emph{formally integrable} if $[\Zf,\Zf]\subset\Zf$;
 \item \emph{of $CR$-type} if $\Zf\cap\overline{\Zf}=\underline{0}$ (the $0$-sheaf);
 \item \emph{almost-$CR$} if  $\Zf$ is of $CR$-type and locally free  
of constant rank;
 \item \emph{quasi-$CR$} if it is of $CR$-type and formally integrable;
\item \emph{$CR$} if $\Zf$ is of $CR$-type, formally integrable
and  locally free of
constant rank.
 \end{itemize}
 A \emph{$\Zf$-manifold} 
 is a real smooth manifold $M$ endowed with a  
$\Zf$-structure.
Since $\CiM$ is a fine sheaf, $\Zf$ can be equivalently described by
the datum of the space $\Zf(M)$ of its global sections.
\end{defn}

When $M$ is a smooth
real submanifold of a complex manifold 
$\Xf$, then
\begin{equation*}
 \Zf(M)=\{Z\in\mathfrak{X}^{\C}(M)\mid Z_p\in{T}^{0,1}_p\Xf
 ,\;\forall p\in{M}\}
\end{equation*} 
is formally integrable. Hence $\Zf(M)$ defines  a quasi-$CR$ structure on $M$,
which is 
$CR$ 
if the dimension of $T^{0,1}_p\Xf\cap\C{T}_pM$ is constant
for $p\in{M}$.
This is always the case
when $M$ is a real hypersurface in $\Xf$. 
\par 
A \textit{complex embedding (immersion)} 
$\phiup:M\hookrightarrow \Xf$
of a quasi-$CR$-manifold $M$ into a complex manifold $\Xf$ 
is a smooth embedding (immersion) for which the $\Zf$-structure on $M$ is the pullback
of the complex structure of $\Xf$:
\begin{equation*}
 \Zf(M)=\{Z\in\mathfrak{X}^{\C}(M)\mid 
 d\phiup(Z_p)\in{T}^{0,1}_{\!\phiup(p)}{\Xf},\;\forall p\in{M}\}.
\end{equation*}

\begin{exam}\label{es1.1}
 Let $M=\{w=z_1\bar{z}_1+i\,z_2\bar{z}_2\}
\subset\C^3_{w,z_1,z_2}={\Xf}.$  We can take the real and
imaginary parts of
$z_1,z_2$ as 
coordinates on $M$, which therefore, as a smooth manifold, 
is diffeomorphic to $\C^2_{z_1,z_2}$. The embedding $M\hookrightarrow\C^3$ yields
the quasi-$CR$ structure 
\begin{equation*}
\Zf(M)=\Ci(M)\left[ z_2\dfrac{\partial}{\partial\bar{z}_1}
+iz_1\dfrac{\partial}{\partial\bar{z}_2}\right]
\end{equation*}
on $M$.
Then $M\setminus\{0\}$ is a $CR$-manifold
of $CR$-dimension $1$ and $CR$-codimension $2$, while all elements
of $\Zf(M)$ vanish at $0\in{M}$.
\end{exam}

A  $\Zf$-manifold $M$ of $CR$-type contains an open dense subset $\ring{M}$
whose connected components are almost-$CR$  for the restriction
of $\Zf$. Likewise, any
quasi-$CR$ manifold $M$ contains an open dense subset
$\ring{M}$ whose connected components are $CR$ manifolds.
\par 
\smallskip
We shall use $\varOmega$ and $\Af$ 
for the sheaves of germs of \textit{complex-valued} and
\textit{real-valued} alterntate forms on $M$ (subscripts  indicate degree of homogeneity).  
Starting with the case of an almost-$CR$ manifold $M$, we introduce  
the notation:
\begin{align*}
 T^{0,1}M&={\bigcup}_{p\in{M}}\big(T^{0,1}_pM=\{Z_p\mid Z\in\Zf(M)\}\big)\subset\C{TM},
 \quad T^{1,0}M=\overline{T^{0,1}M}, \\
 HM&={\bigcup}_{p\in{M}}\big(H_pM=\{\re{Z}_p\mid Z_p\in{T}^{0,1}_pM\}\big)\subset{T}M,\\
 J_{\!{M}} & :H_pM\to{H}_pM,\;\; X_p+iJ_{\!{M}}X_p\in{T}^{0,1}_pM,\;\;\forall X_p\in{H}_pM,\\
 & \qquad\qquad\qquad\qquad\qquad\qquad\qquad\qquad
 \quad\text{(partial complex structure)},\\
 \Hf&=\{\re{Z}\mid Z\in\Zf\},\\
 \piup_M& : TM \to TM/HM\quad\text{(projection onto the quotient),}\\
 \If(M)&=\{\alpha\in{\bigoplus}_{h=1}^\nu\varOmega^h(M,\C)\mid \alpha|T^{0,1}M=0\},\quad
 \text{($\If$ is the \emph{ideal sheaf})},\\
 H^0M&={\bigcup}_{p\in{M}}\big(H^0_pM=\{\xiup\in{T}^*_pM\mid \xiup(H_pM)=\{0\}\}\big)\subset{T}^*M,\\
H^{1,1}M&={\bigcup}_{p\in{M}}\big(H^{1,1}_pM=\text{Convex hull of\; }\{  (Z_p\otimes 
 \bar{Z}_p)\mid Z\in\Zf(M)\}\big),\\
H^{1,1,(r)}M&={\bigcup}_{p\in{M}}\big(H^{1,1,(r)}_pM=\{\tauup\in{H}^{1,1}_pM\mid
\mathrm{rank}\,\tauup=r\}\big).
\end{align*} 
Note that $T^{0,1}M,\; T^{1,0}M,\; HM,\; TM/HM,\; H^0M,\; H^{1,1}M,\; H^{1,1,(r)}M$
define smooth vector bundles because we assumed that the rank $n$ of 
$\Zf$ is constant. This $n$ is called  
the \emph{$CR$-dimension} 
and the difference $k=m-2n$ the \emph{$CR$-codimension} of $M$. \par 
For a general $\Zf$-manifold, we use the same symbols
\par\centerline{$T^{0,1}M, \; T^{1,0}M,\;
 HM,\; TM/HM,\;  
 H^{1,1}M,\; H^{1,1,(r)}M$}
for the closures of 
\par\centerline{
$T^{0,1}\ring{M},\; T^{1,0}\ring{M}, \; H\ring{M}, \; 
T\ring{M}/H\ring{M}, \;  H^{1,1}\ring{M}, \; 
H^{1,1,(r)}\ring{M}$}
in $T^{\C}M$, $T^{\C}M$, $TM$, $TM/HM$, $T^{\C}{M}\otimes_M{T^{\C}{M}}$,
$T^{\C}{M}\otimes_M{T^{\C}{M}}$, respectively.
\begin{exam} For the $M$ in Example~\ref{es1.1}, the fiber $T^{0,1}_pM$
has dimension $1$ at all
points $p$ of $\ring{M}=M\setminus\{0\}$, while
$T^{0,1}_0M=\C[\partial/\bar{z}_1,\partial/\partial\bar{z}_2]$ has dimension $2$. 
By contrast, as we already observed, 
all elements of $\Zf(M)$ vanish at $0$.
\end{exam}
\par\smallskip
If $\Ff$ is a subsheaf of the sheaf of germs of (complex valued) distributions on $M$, 
an element $f$ of $\Ff$ is said to be $CR$ if it satisfies the equations $Zf=0$ for all $Z\in\Zf(M)$.
The $CR$ germs of $\Ff$ are the elements of a sheaf that we denote by $\Ff\!\!\Ot_{\!{M}}$.
We will simply write $\Ot_{\!{M}}$ for $\Ci\!\!\Ot_{\!{M}}$.
\par 
 \smallskip
 We will assume in the rest of this section that $M$ is an almost-$CR$ manifold.
 \par 
The fibers of 
 $H^{1,1}M$ are  closed  convex cones, consisting of 
 the \textit{positive semi-definite} Hermitian symmetric tensors in ${T}^{0,1}M\otimes_M{T}^{1,0}M$.
 The \emph{characteristic bundle} $H^0M$ is the dual of the quotient  $TM/HM$. 
\par 
Let us describe more carefully the bundle structure of
$H^{1,1,(r)}M$.
Set $V=T^{0,1}_pM$ and consider the non-compact Stiefel 
space $\St_r(V)$ of
$r$-tuples 
of linearly independent vectors of $V$. Two different  $r$-tuples $v_1,\hdots,v_r$ and
$w_1,\hdots,w_r$ 
in $\St_r(V)$ 
 define the same $\tauup_p$, i.e. 
satisfy \begin{equation*}
\tauup_p
=v_1\otimes\bar{v}_1+\cdots+v_r\otimes\bar{v}_r=w_1\otimes\bar{w}_1+\cdots+w_r\otimes\bar{w}_r,
\end{equation*}
if and only if there is a matrix $a=(a^i_j)\in\Ub(r)$ (the unitary group of order $r$)
 such that $w_j={\sum}_ja_j^iv_i$. In fact the span of $v_1,\hdots,v_r$ is determined by
the tensor $\tauup_p$, so that $w_j={\sum}_ja_j^iv_i$ for some $a=(a^i_j)\in\GL_r(\C)$
and \begin{align*}
{\sum}_{i=1}^r{w_i\otimes\bar{w}_i}
&={\sum}_{j=1}^r {\sum}_{i,h=1}^r{a}_j^{i}\bar{a}_j^{h} v_i\otimes\bar{v}_h 
= {\sum}_{i,h=1}^r\big({\sum}_{j=1}^r{a}_j^{i}\bar{a}_j^{h}\big) v_i\otimes\bar{v}_h
\end{align*}
shows that $a\in\Ub(r)$. Hence  $H^{1,1,(r)}M$ is the quotient bundle of the non-compact 
complex Stiefel bundle of $r$-frames in $T^{0,1}M$ by the action of the unitary group $\Ub(r)$.
By using the Cartan decomposition \par
\centerline{$\Ub(r)\times \pt(r)\ni(x,X) \xrightarrow{\quad} 
 x\cdot\exp(X)\in\GL_r(\C)$, }\par\noindent
where $\pt(r)$ is
the vector space of Hermitian symmetric ${r\times{r}}$-matrices, we see that  $H^{1,1,(r)}M$
can be viewed as a rank $r^2$ real vector bundle on the Grassmannian $\Gr^r(M)$
of $r$-planes
of $T^{0,1}M$. Thus it is a smooth vector bundle when $M$ is \mbox{almost-$CR$}.
\subsection*{Scalar and vector valued Levi forms}
The map 
\begin{equation}
 Z_p\otimes\bar{Z}_p\longrightarrow
 -\piup_M(i[Z,\bar{Z}]_p),\quad\forall p\in{M},\;\;\forall Z\in\Zf(M),
\end{equation}
extends to a linear map 
\begin{equation}\label{eq1.2}
 \Lf:{H}^{1,1}M \to TM/HM,
\end{equation}
that we 
call the \emph{vector-valued Levi-form}.
\par \smallskip 
To each characteristic co-direction $\xiup\in{H}^0_pM$
we associate the Hermitian quadratic form 
\begin{equation*}
 \Lf_{\xiup}(Z_p,\bar{Z}_p)=\Lf(Z_p\otimes\bar{Z}_p)=-\langle{\xiup}|i[Z,\bar{Z}]_p\rangle, \;\; \forall Z\in\Zf(M).
\end{equation*}
It extends to a convex function on $H^{1,1}_pM$, which is the evaluation by the co-vector $\xiup$ 
of the vector-valued Levi form.  Thus the 
\emph{scalar Levi forms} are 
\begin{equation}
 \Lf_{\xiup}(\tauup)=\xiup(\Lf(\tauup)),\quad\text{for $p\in{M}$,\; $\xiup\in{H}^0_pM$,\; $\tauup\in{H}^{1,1}_pM$.}
\end{equation}
\par\smallskip
The range 
 $\Gamma_pM$ of the vector-valued Levi form 
 is a 
 convex
cone of $T_pM/H_pM$, whose dual cone is 
\begin{equation*}
 \Gamma_p^0M=\{\xiup\in{H}^0_pM\mid \Lf_{\xiup}\geq{0}\}. 
\end{equation*}
Thus we obtain
\begin{lem}\label{lem1.2}
An element $v\in{T}_pM/H_{p}M$ belongs to the 
closure or the 
range of the vector-valued Levi-form 
if and only if 
\begin{equation}  
 \langle{v}|\xiup\rangle \geq{0}, \quad \forall \xiup\in{H}^0_pM\;\;\text{such that}\;\;\Lf_{\xiup}\geq{0}.
 \vspace{-22pt}
\end{equation}
 \qed
\end{lem}
\begin{rmk} Note that $\Gamma_pM$ need not be closed. An example is provided
by the quadric $M=\{\re{z}_3=z_1\bar{z}_1,\; \re{z}_4=\re(z_1\bar{z}_2)\}\subset
\C^4$: the
cone $\Gamma_0M$ is the union of the origin and of an open half-plane.
\end{rmk} 
It is convenient to introduce the notation: \label{kappa}
\begin{equation*}
 \Kf^{(q)}={H}^{1,1,(q)}M\cap\ker\Lf,\; \overline{\Kf}
 ={\bigoplus}_{q\geq{0}}\Kf^{(q)}\!\!\! _, \;\;\;
 \Kf={\bigoplus}_{q>{0}}\Kf^{(q)}\!\!\! _.
\end{equation*}
\begin{defn}
 We call $\Kf$ the \emph{kernel of the Levi form}.
\end{defn}
We note that this definition is at  variance with a notion that appears in the literature
(see e.g. \cite{Dan93}), where the kernel of the Levi form consists of the $(1,0)$-vectors
which are isotropic for all scalar Levi forms. These vectors are related to $\Kf^{(1)}$,
which is trivial in several examples of $CR$ manifolds which are not of hypersurface type
and have a non trivial $\Kf$.
\par\smallskip

Let $\mathpzc{Y}$ be a generalized distribution
of \textit{real} vector fields on $M$ and ${p\in{U}^{\mathrm{open}}\subset{M}}$.
The Sussmann leaf of $\mathpzc{Y}$ through $p$ in $U$ is the set 
$\ell(p;\mathpzc{Y},U)$
of points
$p'$ which are ends of piecewise $\Ci$ integral curves of $\mathpzc{Y}$
starting from $p$ and 
lying in $U$. We know that $\ell(p;\mathpzc{Y},U)$ is always a smooth
submanifold of $U$ (see \cite{Su73}).
\par 
Let $\mathpzc{H}=\{\re{Z}\mid Z\in\Zf\}$. 
A
$\Zf$-manifold $M$ is called 
\emph{minimal} at $p$ if 
$\ell(p; \mathpzc{H},U)$ is an open neighborhood of $p$ for all 
${U}^{\mathrm{open}}\subset{M}$ and 
${p\in{U}}$. (This notion was introduced
in~\cite{Tr90} for embedded $CR$ manifolds.) In the following, 
by a \textit{Sussmann leaf of $\Zf$} we will mean a
Sussmann leaf of $\mathpzc{H}$.

 \par\smallskip
A smooth real submanifold 
$N$ of $M$ (of arbitrary codimension $\ell$) 
is said to be \emph{non-characteristic}, or \emph{generic}, at $p_0\in{N}$, when 
\begin{equation}\label{eq:8.1}
 T_{p_0}N+H_{p_0}M=T_{p_0}M.
\end{equation}
If this holds for all $p\in{N}$, then 
$N$ is a \textit{generic} $CR$ submanifold of $M$, of
type $(n-\ell,k+\ell)$, as $T^{0,1}_pN={T}_p^{\C}N\cap{T}^{0,1}_pN$  and $H^0_pN=H^0_pM\oplus
J_{\!{M}}^*(T_pN)^0$ for all $p\in{N}$.
To distinguish from the Levi form $\Lf$ of $M$, we write
$\Lf^N$ for the Levi form of $N$.
\par 
A Sussmann leaf for $\Zf$ which is not open is \textit{characteristic} at all
points. \par 
More generally, when $\Xi(M)$ is any distribution of complex valued smooth vector fields on $M$,
we say that $N$ is 
\emph{$\Xi$-non-characteristic at $p_0\in{N}$} if
\begin{equation}\label{equnosei}
 T_{p_0}N+\{\re{Z}_{p_0}\mid Z\in\Xi(M)\}=T_{p_0}M.
\end{equation}
In this terminology \textit{non-characteristic} is equivalent to 
$\Zf$-non-characteristic.\par
We note that the  
$\Xi$-non-characteristic points make an open subset of $N$.\par
\section{Hypo-ellipticity and the maximum modulus principle} \label{sec2}
In \cite{NaPo} we proved that, for locally embedded $CR$ manifolds,  
 the hypo-ellipticity of its tangential Cauchy-Riemann
system is equivalent to 
the holomorphic extendability of its $CR$ functions. 
Thus  
hypo-ellipticity 
may be regarded 
as a weak form of pseudo-concavity. 
The regularity of $CR$ distributions
implies a
strong  maximum modulus principle for $CR$ functions (see \cite[Theorem 6.2]{HN00}). 
\begin{prop}\label{pp2.1} Let $M$ be 
a $\Zf$-manifold.
 Assume that all germs of 
 $CR$ distributions  on $M$, that
 are locally $L^2$, are smooth. 
 Then, for every open connected
 subset $\Omega$ of  $M$, we have 
\begin{equation}\label{mmp}
 |f(p)|<{\sup}_{\Omega}|f|,\;\;\forall p\in\Omega,
 \quad\text{for all non constant $f\in\Ot_{\!{M}}(\Omega)$.}
\end{equation}
 \end{prop} 
\begin{proof} We prove that an $f\in\Ot_{\!{M}}(\Omega)$ for which $|f|$ 
attains  a maximum value at some inner
point $p_0$ of $\Omega$ is constant.
Assume that $p_0\in\Omega$ and
$|f(p_0)|=\sup_{\Omega}|f|$. If $f(p_0)=0$, then $f$ is constant and equal to zero on $\Omega$. \par 
Assume  that $f(p_0)\neq{0}$. After rescaling, we can make $f(p_0)=|f(p_0)|=1$. \par 
Let $E$
be the space $\Ot_{\!{M}}(\Omega)$ endowed with the $L^2_{loc}$ topology. 
By the  hypo-ellipticity assumption, $E$ is Fr\'echet. Then,
by Banach open mapping theorem, the identity map $E\to\Ot_{\!{M}}(\Omega)$
is an isomorpism of topological vector spaces. 
In particular, for all compact neighborhoods $K$ of $p_0$ in $\Omega$ there is a constant
$C_K>0$ such that 
\begin{equation*}
 |u(p_0)|^2\leq{C}_K\int_K |u|^2 d\lambda,\quad \forall u\in\Ot_{\!{M}}(\Omega).
\end{equation*}
Applying this inequality to $f^{\nu}$, we obtain that 
\begin{equation*}
 1\leq\int_K |f|^{2\nu}d\lambda \leq\int_K d\lambda.
\end{equation*}
The sequence $\{f^\nu\}$ is compact in $\Ot_{\!{M}}(\Omega)$, because, by the hypo-ellipticity
assumption 
and the Ascoli-Arzel\`a theorem, restriction to a relatively compact subset of $CR$ functions
is a compact map. Hence we can extract from $\{f^\nu\}$ a sequence that converges to a 
$CR$ function $\phi$, which is non-zero because 
it has a positive square-integral on every compact
neighborhood of $p_0$. We note now that $|\phi|$ is continuous, and takes only the values $1$,
at points where $|f|=1$, 
and $0$ at points where $|f|<1$. Since $\phi\neq{0}$, we have $|\phi|\equiv{1}$ on $\Omega$
and hence $|f|\equiv{1}$ on $\Omega$. 
By applying the preceding argument
to $p\to\tfrac{1}{2}(1+f(p))$, we obtain that  $|1+f(p)|=2$ on $\Omega$. Hence
$\re{f}\equiv{1}$, which yields $f\equiv{1}$, on $\Omega$.  
\end{proof}
Under the assumptions of 
Proposition~\ref{pp2.1}, a $CR$ function $f\in\Ot_{\!{M}}(\Omega)$ is constant on a neighborhood of
any point where $|f|$ attains a local maximum.\par 

Then we have 
\begin{prop}
 Assume that 
\begin{enumerate}
 \item[$(i)$] all germs of $CR$ distribution on $M$ are smooth;
 \item[$(ii)$] the weak unique-continuation principle for $CR$ functions is valid on $M$.
\end{enumerate}
Then any $CR$ function $f,$ defined on a connected open subset $\Omega$ of $M$,
for which $|f|$ attains a local maximum at some point of $\Omega$, is constant.\qed
\end{prop}
We recall that the weak unique-continuation property $(ii)$ means that a $CR$ function
$f\in\Ot_{\!{M}}(\Omega)$ which is zero on an open subset $U$ of $\Omega$ is zero
on the connected component of $U$ in $\Omega$. \par
\begin{defn}
 We say that $M$ has property $(H)$ if $(i)$ holds,
 and property $({W\!{U}\!{C}})$ if $(ii)$ holds.  We say that $(H)$ (or $(W\! U\! C)$) 
 holds at $p$
 if it holds when $M$ is substituted by a
 sufficiently small open neighborhood of $p$ in $M$. \end{defn}
 For a locally $CR$-embeddable $CR$ manifold $M$
the implication $(H)\Rightarrow(W\! U\! C)$ is a consequence of~\cite{NaPo}.
In fact, 
$(H)$ implies \textit{minimality}, which implies $(W\! U\! C)$ when $M$ is locally $CR$-embeddable
(see \cite{Tr90, Tu89}). In fact, in this case $(H)$ implies
 the \textit{strong unique continuation principle}  for $CR$ functions.
\begin{prop} \label{propo2.3}
 Assume that $M$ is a $CR$ submanifold of a complex manifold ${\Xf}$ and that $M$ has
 property $(H)$. Then a $CR$ function, defined on a connected open subset $\Omega$ of $M$ and
 vanishing to infinite order at a point $p_0$ of $\Omega$ 
 is identically zero in $\Omega$.
\end{prop} 
\begin{proof} Let $f\in\Ot_{\!{M}}(\Omega)$. 
 It is sufficient to prove that the set of points where $f$ vanishes to infinite order is open in $\Omega$.
 This reduces the proof to a local statement, allowing us to assume that the embedding
 $M\hookrightarrow{\Xf}$ is generic. By \cite{NaPo}, any $CR$ function 
 $f$  extends to a holomorphic function $\tilde{f}$,
 defined on a connected open neighborhood $U$ of $p$ in ${\Xf}$. By the assumption
 that $M\hookrightarrow{\Xf}$ is generic, $\tilde{f}$ is uniquely determined 
 by the Taylor series of $f$ at $p$ in any coordinate chart, 
 and thus
 vanishes to infinite order at a point
 $p'\in{U}\cap\Omega$ if and only if $f$ does. Hence $f$ vanishes to infinite order at $p$ if and
 only if $\tilde{f}$ vanishes on $U$, and this is equivalent to the fact that $f$
vanishes identically on $U\cap\Omega$. The proof is complete.
\end{proof} 
\par\smallskip 
When $M$ is \textit{not} locally embeddable, there should be smaller local rings of $CR$ functions, so 
that in fact properties of regularity and  unique continuation should even be more likely true. 
Let us shortly discuss this issue.
Set
\begin{equation*}
 {T^*_p}^{1,0}M=\{\zetaup\in\C{T}^*_pM\mid \zetaup(Z)=0,\;\forall Z\in{T}^{0,1}_pM\}.
\end{equation*}
Note that, with 
\begin{equation*}
 {T^*_p}^{0,1}M=\overline{{T^*_p}^{1,0}M}
 =\{\zetaup\in\C{T}^*_pM\mid \zetaup(\bar{Z})=0,\;\forall Z\in{T}^{0,1}_pM\},
\end{equation*}
the intersection 
\begin{equation*}
 {T^*_p}^{1,0}M\cap{T^*_p}^{0,1}M=\C{H}^0_pM
\end{equation*}
is the complexification of the fiber of the characteristic bundle, and therefore different from zero, unless
$\Zf$ is an almost complex structure. Differentials of smooth $CR$ functions are sections of the
bundle ${T^*}^{1,0}M$. Thus,
for a fixed $p$, we can consider the map 
\begin{equation}\label{eq:2.2}
 \Ot_{\!{M},p}\ni{f}\longrightarrow df(p)\in{T^*_p}^{1,0}M.
\end{equation}
Clearly we have 
\begin{lem}
 A necessary and sufficient condition for $M$ to be locally $CR$-embed\-dable at $p$ is that
 \eqref{eq:2.2} is surjective.\qed
\end{lem}
We can associate to the map \eqref{eq:2.2} a pair $(\nu_p,k_p)$ of nonnegative integers, with\par
\centerline{
$k_p=\dim_{\C}\{df(p)\mid{f}\in\Ot_{\!{M},p}\}\cap{\C}H^0_pM$, \;\; and\;\;
$\nu_p+k_p=\dim_{\C}\{df(p)\mid{f}\in\Ot_{\!{M},p}\}$.}
The numbers $\nu_p$ and $\nu_p+k_p$ are upper semicontinuous functions of 
$p$
and hence locally constant on a dense open subset $\ring{M}$ of $M$. 
Thus we can introduce 
\begin{defn}
 We call \emph{generic} the points of the open dense subset $\ring{M}$ of
 $M$, where $\nu_p$ and $\nu_p+k_p$ are locally constant. 
\end{defn}

\begin{prop}\label{prop2.5}
 Assume that $M$ has property $(H)$. Then the strong unique continuation
principle is valid at generic points $p_0$ of $M$.
This means that $f\in\Ot_{\!{M},p_0}$ is the zero germ if 
 and only if vanishes to infinite
 order at $p_0$. \par
 Moreover, there are finitely many germs $f_1,\hdots,f_\mu\in\Ot_{\!{M},p_0}$,
 vanishing at $p_0$,  
 such that, for 
 every $f\in\Ot_{\!{M},p_0}$ we can find $F\in\Ot_{\C^{\mu},0}$ such that $f=F(f_1,\hdots,f_\mu)$.
\end{prop} 
\begin{proof} By the assumption that $p_0$ is generic,
we can fix a connected 
open neighborhood $U$ of $p_0$ in $M$ and functions $f_1,\hdots,f_\mu\in\Ot_{M}(U)$,
vanishing at $p_0$, 
such that $df_1(p)\wedge\cdots\wedge{df_\mu(p)}\neq{0}$ for all $p\in{U}$ and $df_1(p),\hdots,
df_\mu(p)$ generate the image of \eqref{eq:2.2} for all $p\in{U}$. Then,
by shrinking $U$, if needed, we can assume that
\begin{equation*}
 \phiup:U\ni{p}\longrightarrow (f_1(p),\hdots,f_\mu(p))\in{N}\subset\C^{\mu}
\end{equation*}
is a smooth real vector bundle on 
a generic
$CR$ submanifold $N$ of $\C^\mu$, of $CR$-dimension $\nu_{\!{p_0}}$
and $CR$-codimension $k_{p_0}$. \par
In fact, we can assume that $\re{df_1},\hdots,\re{df_\mu},\im(df_1),\hdots,\im(df_\nu)$
are linearly independent on $U$. We can fix local coordinates $x_1,\hdots,x_m$ 
centered at $p_0$ with $x_1,\hdots,x_{\mu+\nu}$ equal to $\re{f}_1,\hdots,\re{f_\mu},\im{f_1},
\hdots, \im{f}_\nu$. By the assumption, in these local coordinates $\im{f}_{\nu+1},\hdots,\im{f}_\mu$
are smooth functions of $x_1,\hdots,x_{\mu+\nu}$ and this yields a parametric representation of
$N$ as a graph of $\C^\nu\times\R^{\mu-\nu}$ in $\C^\mu$, which is therefore locally a 
generic $CR$-submanifold
of type $(\nu,\mu-\nu)$ of $\C^\mu$.  The map $\phiup:U\to{N}$ is $CR$ and therefore
the pullback of germs of continuous $CR$ function on $N$ define
germs of continuous $CR$ function on $M$. 
If $M$ has property $(H)$, then the $\Ci$ regularity of their pullbacks implies the $\Ci$ regularity of the germs
on $N$. Thus $N$ also has property $(H)$, and, since it is embedded in $\C^\mu$, by \cite{NaPo},
all $CR$ functions on an open neighborhood $\omega_0$ of $0$ in 
$N$ are the restriction of homomorphic functions on a full open neighborhood $\tilde{\omega}_0$ of
$0$ in $\C^\mu$, with $\omega_0=\tilde{\omega}_0\cap{N}$. Since $f_i=\phiup^*(z_i)$ for the
holomorphic coordinates $z_1,\hdots,z_\mu$ of $\C^\mu$, we obtain that all germs of $CR$ functions
at $p_0\in{M}$ are germs of holomorphic functions of $f_1,\hdots,f_\mu$. 
This clearly implies the validity at $p_0$ of the strong unique continuation principle.
The proof is complete.
\end{proof}

\section{The kernel of the Levi form and the \textsl{(H)} property} \label{s3}
To a finite set $Z_1,\hdots,Z_r$ of vector fields in $\Zf(M)$ we associate 
the  real valued
vector field 
\begin{equation}\label{e2.1}
Y_0=\frac{1}{2i}\, {\sum}_{j=1}^r[Z_j,\bar{Z}_j].
\end{equation}
\par
Any $CR$ function $u$ on $M$ satisfies
the degenerate-Schr\"odinger-type equation 
\begin{gather}
 S_{\!}u=0, \qquad \text{with}\\ \label{e2.2}
  S_{\!}= -iY_0 +\tfrac{1}{2}{\sum}_{j=1}^r(Z_j\bar{Z}_j+\bar{Z}_j{Z}_j)= -iY_0+{\sum}_{j=1}^{2r}X_j^2,
\end{gather}
where $X_j=\re{Z}_j$, 
$X_{j+r}=\im{Z}_{j}$ for $1\leq{j}\leq{r}.$
In fact, by \eqref{e2.1}, we have 
\begin{equation*}
 S_{\!}=\tfrac{1}{2}{\sum}_{i=1}^r\bar{Z}_jZ_j,
\end{equation*}
and  thus
 the operator $S$ belongs to the left ideal, in the ring of scalar linear partial differential operators
 with complex smooth coefficients,
 generated by $\Zf(M)$.
\par 
We note that $S_{\!}$ is of the second order, with a real principal part which is uniquely determined
by $\tauup=Z_1\otimes\bar{Z}_1+
\cdots+Z_r\otimes\bar{Z}_r\in\Gamma(H^{1,1}M)$, 
while a different choice of the 
$Z_j$'s would yield a new $Y_0'$, differing form $Y_0$ 
by the addition
of a linear combination of $X_1,\hdots,X_{2r}$. 
\par 

\smallskip
\par 
If we assume that $\tauup\in\ker(\Lf)$, then 
\begin{equation}\label{1.4a}
{\sum}_{i=1}^r[Z_j,\bar{Z}_j]=\bar{L}_0-{L}_0 
\end{equation}
for some $L\in\Zf(M)$, which 
is uniquely determined by $\tauup$ modulo 
a linear combination with $\Ci$ coefficients of $Z_1,\hdots,Z_r$.
Thus 
the distributions of \textit{real} vector fields \begin{equation} 
\label{3.5eq}
\left\{
\begin{aligned}
 \Qf_1(\tauup)&=\langle \re{Z}_1,\hdots,\re{Z}_r,\im{Z}_1,\hdots,\im{Z}_r\rangle,\\
 \Vf_1(\tauup)&=\Li(\Qf_1(\tauup)),\\
 \Vf_2(\tauup)&=\Li(\Qf_1(\tauup)+
 \re{L_0}
 ),
\end{aligned}\right.\end{equation}
are uniquely determined by $\tauup$ and $\Zf$.
By %
$\Li(...)$ 
we %
indicate 
the formally integrable
distribution of real vector fields, which is  
generated by the elements of the set inside the parentheses 
and their iterated commutators.
Note that $\Vf_1(\tauup)\subseteq\Vf_2(\tauup)$
and, while  $Y_0=\im{L}_0\in\Vf_1(\tauup)$, 
the vector field
$X_0=\re{L}_0$ 
may not belong to $\Vf_1(\tauup)$.
We also introduce, for further reference, the distributions of \textit{complex} vector fields 
\begin{equation}\label{theta3.6} \begin{cases}
 \Theta(\tau)=\langle Z_1,\hdots,Z_r\rangle\;\; \text{and}\;\; 
 \Theta={\bigcup}_{\tauup\in\Kf}\Theta(\tauup),\\
 \tilde{\Theta}(\tau)= \Theta(\tau)+
 \langle L_0 \rangle\;\; \text{and}\;\; \tilde{\Theta}={\bigcup}_{\tauup\in\Kf}\tilde{\Theta}(\tauup)
 \end{cases}
\end{equation}
\par 
When there is a
$\tauup\in\Kf(\Omega^{\mathrm{open}}),$ 
we utilize \eqref{1.4a} to show that the 
real and imaginary parts of $CR$ functions or distributions on 
$\Omega\subset{M}$ 
are solutions of 
a \textit{real degenerate-elliptic scalar second order differential equation}.
Indeed, 
if 
$f$ is a $CR$ function, or distribution, 
in 
$\Omega$, 
then 
\begin{equation*}
 L_0f=0,\quad Z_jf=0 \Longrightarrow (\bar{L}_0+{L}_0)f={\sum}_{i=1}^r({Z_j\bar{Z}_j+\bar{Z}_jZ_j})f.
 \end{equation*}
 This is  a consequence of the algebraic 
 identities 
\begin{equation}\label{1.4b}
 \tfrac{1}{2}\left\{{\sum}_{i=1}^r({Z_j\bar{Z}_j+\bar{Z}_jZ_j})
 -(\bar{L}_0+{L}_0)\right\}=
 {\sum}_{i=1}^r{\bar{Z}_jZ_j}-L_0={\sum}_{i=1}^rZ_j\bar{Z}_j-\bar{L}_0.
\end{equation}\par 
It terms of the \textit{real} vector fields 
$X_0=\re{L}_0$ and 
$X_j=\re{Z}_j$, ${X_{r+j}=\im{Z}_j}$, for $1\leq{j}\leq{r}$, 
the linear partial differential operator of \eqref{1.4b} is 
\begin{equation}\label{2.2} 
P_{\!\tauup}
= - X_0+{\sum}_{i=1}^{2r}X_j^2,
\end{equation}
which has real valued coefficients and
is degenerate-elliptic according to \cite{Bo69}. 
Thus \textsl{the real and imaginary parts of a $CR$ function, or distribution, both
satisfy the homogeneous equation $P_{\!\tauup}\phiup=0$.} \par 

Actually, 
$P_{\!\tauup}$ is independent of the choice of 
$Z_1,\hdots,Z_r$
in the representation of $\tauup$, 
as we will later show by
Proposition~\ref{p4.6}, 
by representing 
$P_{\!\tauup}$ in terms of the $dd^c$ operator~on~$M$. 
We also have (see \cite{HN03}):
\begin{lem} If $u\in\Ot_{\!{M}}(\Omega)$, then \begin{equation}
P_{\!\tauup}|u|\geq{0},\quad \text{on \;$\Omega\cap\{u\neq{0}\}$.}
\end{equation}
\end{lem} \begin{proof} 
On a neighborhood of a point where $u\neq{0}$ we can consistently define 
a branch of $\log{u}$.
This still is a $CR$ function, and from the previous observation it follows that 
$P_{\!\tauup}(\log|u|)=P_{\!\tauup}(\re\log{u})=0$ on  $\Omega\cap\{u\neq{0}\}$. 
Hence \begin{align*}
 P_{\!\tauup}|u|&= P_{\!\tauup}\exp(\log|u|)= |u|\left(P_{\!\tauup}(\log|u|)+{\sum}_{i=1}^r |Z_j(\log{|u|})|^2\right)
\\
&=|u| {\sum}_{i=1}^r |Z_j(\log{|u|})|^2\geq{0}
\end{align*}
there.
\end{proof}


We can use the treatment of the generalized Kolmogorov equation in \cite[\S{22.2}]{Hor85}
to slightly improve  the regularity result of \cite[Corollary 1.15]{AHNP08a}. Let us set 
\begin{equation} \Vf_2=\Li\left( {\bigcup}_{\tauup\in\Kf}\Vf_2(\tauup)\right),\qquad 
 \Yf=\Li(\Vf_2;\Hf), 
\end{equation}
where we use $\Li(\Vf_2;\Hf)$ for the $\Vf_2$-Lie module generated by $\Hf$, which consists 
of the linear combinations, with smooth real coefficients, of the elements of $\Hf$ and their
iterated commutators with elements of $\Vf_2$: 
\begin{equation}
 \Li(\Vf_2;\Hf)=\Hf+[\Vf_2,\Hf]+[\Vf_2,[\Vf_2,\Hf]]+[\Vf_2,[\Vf_2,[\Vf_2,\Hf]]]+\cdots
\end{equation}
Note that $\Vf_2\subset\Li(\Vf_2;\Hf)$ and that 
both $\Vf_2$ and $\Yf$ are fine sheaves. 
\begin{thm}\label{thm3.3}
 $M$ has property $(H)$ at all points $p$ where $\{Y_p\mid Y\in\Yf(M)\}=T_pM$.
\end{thm}
Before proving the theorem, let us introduce some notation. For $\epsilon>0$ we 
denote by $\Sub_{\epsilon}(M)$ the set of real vector fields $Y\in\mathfrak{X}(M)$
such that, for every $p\in{M}$ there is 
a neighborhood 
${U^{\mathrm{open}}\Subset{M}}$ of $p$,
a constant 
 $C\geq{0}$, $\tauup_1,\hdots,\tauup_h\in\Kf(M)$ and complex vector fields 
$Z_1,\hdots,Z_\ell\in\Zf(M)$ such that 
\begin{equation}\label{sob3.11}
\|Yf\|_{\epsilon-1}\leq C\big({\sum}_{j=1}^h\|P_{\!\tauup_j}{f}\|_0+
{\sum}_{i=1}^\ell\|Z_jf\|_0+\|f\|_0\big),\;\;\forall f\in\Cic(U).
\end{equation}
The Sobolev norms of real order (and integrability two)
in \eqref{sob3.11}
are of course computed after 
fixing a Riemannian metric on $M$.
Different choices of the metric yield equivalent norms
(see e.g. {\cite{AHNP08a, heb96}} for technical details).
Beware that
the $Z_j$ in the right hand side of 
\eqref{sob3.11}
are not required to be 
related to those entering the definition of the
$P_{\!\tauup_j}$'s.
Set 
\begin{equation}
 \Sub(M)={\bigcup}_{\epsilon>0}\Sub_{\epsilon}(M).
\end{equation}
Theorem~\ref{thm3.3} will follow from the inclusion $\Yf(M)\subset\Sub(M)$.\par 
The following
Lemmas~\ref{lem3.4} and \ref{lem3.5} were proved in \cite{AHNP08a, HN00}. 
\begin{lem}\label{lem3.4}
 If $\tauup\in\Kf(M)$ and $\Pt=-X_0+{\sum}_{i=1}^{2r}X_i^2$, then 
 $X_1,\hdots,X_{2r}\in\Sub_1(M)$ and, 
 for every $U^{\mathrm{open}}\Subset{M}$
 there are a constant $C>0$ and $Z_1,\hdots,Z_\ell\in\Zf(M)$ such that 
\begin{equation} \vspace{-18pt}
 {\sum}_{i=1}^{2r}\|X_i{f}\|_0\leq C\left(\|f\|_0+{\sum}_{j=1}^\ell\|Z_j{f}\|_0\right),\;\;\forall f\in\Cic(U).
\end{equation} \qed
\end{lem}
Set $\Vf_1=\Li\left({\bigcup}_{\tauup\in\Kf}\Vf_1(\tauup)\right)$ and 
\begin{equation*}
 \Li(\Vf_1;\Hf)=\Hf+[\Vf_1,\Hf]+[\Vf_1,[\Vf_1,\Hf]]+[\Vf_1,[\Vf_1,[\Vf_1,\Hf]]]+\cdots.
\end{equation*}

\begin{lem}\label{lem3.5} We have the inclusion
$\Li(\Vf_1;\Hf)\subset\Sub.$ \qed
\end{lem}
To prove Theorem~\ref{thm3.3}, we add the following lemma. 
\begin{lem}
 Let $\tauup\in\Kf(M)$, with $P_{\!\tauup}=-X_0+{\sum}_{i=1}^{2r}X_i^2$. Then 
\begin{equation}\label{sub3.15}
 [X_0,\Sub_{\epsilon}(M)]\subset\Sub_{\epsilon/4}(M).
\end{equation}
\end{lem} 
\begin{proof} Let $\Qt=\Pt+c$, for a suitable nonnegative real constant $c$, to be precised later.
We decompose $\Qt$ into the sum $\Qt=\Qt'+i\Qt''$, where 
$\Qt'=\tfrac{1}{2}(\Qt+\Qt^*)$ and $\Qt''=\tfrac{1}{2i}(\Qt-\Qt^*)$ are selfadjoint.
In particular, $\Qt^*=\Qt'-i\Qt''$. 
We can rewrite $\Qt'$ as a sum $\Qt'=-{\sum}_{j=1}^{2r}X_j^*X_j+iT+c$, for a p.d.o. 
$T$ of order $\leq{1}$, whose principal part of order $1$ 
is a linear combination with $\Ci$ coefficients of $X_1,\hdots,X_{2r}$. 
Moreover, we note that $\Pt-\Pt^*=\Qt-\Qt^*$.
The advantage in dealing with $\Qt$ instead of $\Pt$ is that,
for
$c$ positive and sufficiently large,  
\begin{equation}\label{star}\tag{$*$} (\Qt{f}|f)_0=
 (\Qt'f|f)_0\geq{0},\;\;\forall f\in\Cic(U).
\end{equation}
This is the single requirement for
our choice of $c$. \par 
In \cite{AHNP08a} it was shown that $[X_i,\Sub_{\epsilon}]\subset\Sub_{\frac{\epsilon}{2}}$ for 
$i=1,\hdots,2r$ and 
all $\epsilon>0$. 
Then \eqref{sub3.15}
is equivalent to the inclusion ${[\Qt'',\Sub_{\epsilon}]\subset\Sub_{\frac{\epsilon}{4}}}$.\par 
 Let $Y\in\Sub_{\epsilon}(M)$ and $U^{\mathrm{open}}\Subset{M}$.
We need to estimate 
 $\|\,[\Qt'',Y]f\|_{\frac{\epsilon}{4}-1}$ for ${f\in\Cic(U)}$.
 Let $A$ be any properly supported 
 pseudodifferential operator of order $\frac{\epsilon}{2}-1$. 
 We have 
\begin{align*} i([\Qt'',Y]f|Af)&=((\Qt'-\Qt^*)Yf|Af)_0+((\Qt-\Qt')f|Y^*Af)_0\\
&=(\Qt'{Yf}|Af)_0-(Yf|\Qt{A}f)_0+ (\Qt{f}|Y^*Af)_0-(\Qt'f|Y^*Af)_0. 
\end{align*}\par
While estimating the summands in the last expression, we shall indicate by $C_1,C_2,\hdots$ 
positive constants
independent of 
the choice of  $f$ in $\Cic(U)$.
\par
Let us first consider the second and third summands. We have 
\begin{align*}
 \left|(Yf|\Qt{A}f)_0
 \right| & \leq \|Yf\|_{\epsilon-1}\|\Qt{Af}\|_{1-\epsilon}\leq \|Yf\|_{\epsilon-1}\left( 
 \|A\Qt{f}\|_{1-\epsilon}+\|\, [A,\Qt]{f}\|_{1-\epsilon}
 \right)\\
 &\leq C_1 \|Yf\|_{\epsilon-1}\left( \|\Qt{f}\|_{-\frac{\epsilon}{2}} + \left\|\, \left[A,{\sum}_{j=1}^{2r}X_j^2\right]{f}
 \right\|_{1-\epsilon}+ \|f\|_{-\frac{\epsilon}{2}}\right).
\end{align*}
We have 
\begin{align*}
 \left[A,{\sum}_{j=1}^{2r}X_j^2\right]=-
 {\sum}_{j=1}^{2r}\left(2[X_j,A]X_j+[X_j,[X_j,A]]\right).
\end{align*}
Since $[X_j,A]$ and $[X_j,[X_j,A]]$ have order $\tfrac{\epsilon}{2}-1$, 
and $\Pt$ and $\Qt$ differ by a constant,
we obtain 
\begin{equation*}
  \left|(Yf|\Qt{A}f)_0
 \right| \leq C_2\, \|{Y}f\|_{\frac{\epsilon}{2}-1}
 \left(\|\Pt{f}\|_{-\frac{\epsilon}{2}} + \|f\|_{-\frac{\epsilon}{2}}+{\sum}_{j=1}^{2r}\|X_jf\|_{-\frac{\epsilon}{2}}
 \right). 
\end{equation*}
Analogously, for the third summand we have, since $(Y+Y^*)$ has order zero,  
\begin{align*}
 |(\Qt{f}|Y^*Af)_0|&\leq \|\Qt{f}\|_0\left(\|AY^*f\|_0+\|\,[Y^*,A]f\|_0\right)\\
 &\leq C_2 \left(\|\Pt{f}\|_0+\|f\|_0\right)\left(  \|Yf\|_{\frac{\epsilon}{2}-1}+\|f\|_{\frac{\epsilon}{2}-1} \right).
\end{align*}
Next we consider 
\begin{align*}
 |(\Qt'Yf|Af)_0|=|(Yf|\Qt'Af)|\leq |(Yf|A\Qt'{f})_0|+|(Yf|[\Qt',A]f)_0|.
\end{align*}
Let us first estimate the second summand in the last expression. \par 
We have $\Qt'={\sum}_{i=1}^{2r}X_i^2+R'_0$ for a first order p.d.o. $R'_0$ 
whose principal part is a linear combination
of $X_1,\hdots,X_{2r}$. Hence 
\begin{align*}
 [\Qt',A]=[R_0',A]+{\sum}\left( 2[X_i,A]X_i+[X_i,[X_i,A]]\right),
\end{align*}
with pseudodifferential 
operators $[R'_0,A]$, $[X_i,A]$, $[X_i,[X_i,A]]$ of order $\leq(\tfrac{\epsilon}{4}-1)$.
Thus we obtain
\begin{align*}
 |(Yf|[\Qt',A]f)_0|\leq C_3 \|Yf\|_{\epsilon-1}\left( \|f\|_{-\frac{\epsilon}{4}}+{\sum}_{j=1}^{2r}\|X_jf\|_{-\tfrac{\epsilon}{4}}\right).
\end{align*}
\par 
Because of \eqref{star}, 
we have the Cauchy inequality 
\begin{equation*}
 |(\Qt'{f}_1|f_2)|\leq \sqrt{(\Qt'f_1|f_1)\,(\Qt'f_2|f_2)},\quad \text{for $f_1,f_2\in\Cic(U)$.}
\end{equation*}
Hence 
\begin{gather*}
|(Yf|A\Qt'{f})_0|^2=
  |(\Qt'{f}|A^*Yf)_0|^2\leq (\Qt'A^*Yf|A^*Yf)_0(\Qt'f|f)_0, 
 \\ 
  |(\Qt'{f}|Y^*Af)_0| \leq (\Qt'Y^*Af|Y^*Af)_0(\Qt'f|f)_0.
\end{gather*}
We have, for the second factor on the right hand sides, 
\begin{equation*}
 (\Qt'{f}|f)_0=(\Qt{f}|f)_0\leq \|\Qt{f}\|_0\|f\|_0\leq (\|\Pt{f}\|_0+|c|\|f\|_0)\|f\|_0.
\end{equation*}
Let us estimate the first factors. We get
\begin{align*}
 (\Qt'A^*Yf|A^*Yf)_0&=(\Qt{A}^*Yf|A^*Yf)\leq \|\Qt{A}^*Yf\|_{-\frac{\epsilon}{2}}\|A^*Yf\|_{\frac{\epsilon}{2}}
 \\
 &
 \leq \|A^*Yf\|_{\frac{\epsilon}{2}}\left( \| A^*Y\Qt{f}\|_{-\frac{\epsilon}{2}}+\|\,[A^*Y,\Qt]{f}\|_{-\frac{\epsilon}{2}}
 \right)\\
 & \leq C_3\|Yf\|_{\epsilon-1}\left(\|\Qt{f}\|_0+ \|\,[A^*Y,\Qt]{f}\|_{-\frac{\epsilon}{2}}
 \right).
\end{align*}
We need to estimate the second summand inside the parentheses in the last expression. We note that
\begin{align*}
 [A^*Y,\Qt]=[A^*Y,\Pt]=
 -[A^*Y,X_0]+{\sum}_{j=1}^{2r}\left(2[A^*Y,X_j]X_j+[X_j,[A^*Y,X_j]]\right).
\end{align*}
Since the operators $[A^*Y,X_0]$, $[A^*Y,X_j]$, $[X_j,[A^*Y,X_j]]$ have order $\tfrac{\epsilon}{2}$,
we obtain 
\begin{equation*}
  \|\,[A^*Y,\Qt]{f}\|_{-\frac{\epsilon}{2}}\leq C_4 \left(\|f\|_0+{\sum}_{j=1}^{2r}\|X_jf\|_0\right).
\end{equation*}
Finally, 
\begin{align*}
 (\Qt'Y^*Af|Y^*Af)_0&=(\Qt{Y}^*Af|Y^*Af)\leq \|Y^*Af\|_{\frac{\epsilon}{2}}\left(
 \|Y^*A\Qt{f}\|_{-\frac{\epsilon}{2}}+\|\,[\Qt,Y^*A]f\|_{-\frac{\epsilon}{2}}\right)\\
 &\leq C_5\|Y^*Af\|_{\frac{\epsilon}{2}}\left( \|\Qt{f}\|_0+ \|\,[\Qt,Y^*A]f\|_{-\frac{\epsilon}{2}}\right).
\end{align*}
Since 
\begin{align*}
  [Y^*A,\Qt]=[Y^*A,\Pt]=
  -[Y^*A,X_0]+{\sum}_{j=1}^{2r}\left(2[Y^*A,X_j]X_j+[X_j,[Y^*A,X_j]]\right)
\end{align*}
 and the operators $[Y^*A,X_0]$, $[Y^*A,X_j]$, $[X_j,[Y^*A,X_j]]$ have order $\frac{\epsilon}{2}$, 
 we obtain that 
\begin{equation*}
\|\,[\Qt,Y^*A]f\|_{-\frac{\epsilon}{2}}\leq C_6   \left(\|f\|_0+{\sum}_{j=1}^{2r}\|X_jf\|_0\right).
\end{equation*}
 Moreover, 
\begin{align*}
 Y^*A=-AY+(Y+Y^*)A+[A,Y],
\end{align*}
with $\{(Y+Y^*)A\! +\! [A,Y]\}$ of order 
 $\leq(\frac{\epsilon}{2}-1)$, because $Y+Y^*$ has order $0$. Hence 
\begin{equation*}
 \|Y^*Af\|_{\frac{\epsilon}{2}}\leq C_7\left(\|Yf\|_{\epsilon-1}+\|f\|_0\right).
\end{equation*}
Putting all these inequalities together, we conclude that 
\begin{equation*}
| ([X_0,Y]f|Af)_0|\leq C_8\left( \|f\|_0^2+\|Yf\|_{\epsilon-1}^2+
\|\Pt{f}\|_0^2+{\sum}_{j=1}^{2r}\|X_jf\|_0^2\right),\quad
 \forall f\in\Cic(U).
\end{equation*}
By taking $A=\Lambda_{\frac{\epsilon}{2}-1}[X_0,Y]$ for an elliptic properly supported pseudodifferential
operator $\Lambda_{\frac{\epsilon}{2}-1}$ 
of order $\frac{\epsilon}{2}-1$, we deduce that 
\begin{equation*}
 \|[X_0,Y]f\|_{\frac{\epsilon}{4}-1}\leq{C}_9\left(\|f\|_0+\|Yf\|_{{\epsilon}-1}+\|\Pt{f}\|_0+
{ \sum}_{i=1}^{2r}\|X_i f\|_0\right)
\end{equation*}
and therefore, since $X_1,\hdots,X_{2r}\in\Sub_1(M)$ and $Y\in\Sub_{\epsilon}(M)$, that
$[X_0,Y]\in\Sub_{\frac{\epsilon}{4}}$. 
\end{proof} 
\begin{cor}\label{cor3.6}
 We have 
\begin{equation}\vspace{-18pt}
 \Li(\Vf_2;\Sub)\subset\Sub.
\end{equation} \qed
\end{cor}
\begin{proof}[Proof of Theorem~\ref{thm3.3}] By the assumption, 
$\{Y_q\mid Y\in\Sub(M)\}=T_qM$ for all $q$ in an open neighborhood  of $p$ in $M$. 
Thus there are $p\in{U}^{\mathrm{open}}\Subset{M}$, $\tauup_1,\hdots,\tauup_h\in\Kf(M)$,
$Z_1,\hdots,Z_\ell\in\Zf(M)$ and $C>0$ such that 
\begin{equation}
 \|f\|_{\epsilon}\leq C\left (\|f\|_0+{\sum}_{j=1}^h\|P_{\tauup_j}f\|_0+{\sum}_{i=1}^\ell\|Z_if\|_0\right),\quad
 \forall f\in\Cic(U).
\end{equation}
Let $\Ptj=-X_{0,j}+{\sum}_{s=1}^{2r_j}X_{s,j}^2$, with $Z_{s,j}=X_{s,j}+iX_{s+r_j,j}\in\Zf(M)$ for $1\leq{s}\leq{r}_j$,
and let $Z_{0,j}$ be the vector field in $\Zf(M)$ with $\re{Z}_{0,j}=X_{0,j}$.  
If  $A$ is 
a properly supported pseudodifferential operator, then
\begin{equation*}
 [\Ptj,A]=-[X_{0,j},A]+{\sum}_{s=1}^{2r_j}\left(2X_{s,j},[X_{s,j},A]+[[X_{s,j},A],X_{s,j}]\right).
\end{equation*}
If $A$ has order $\delta$, and is zero outside a compact subset $K$ of $U$, and $\chiup$ is a smooth
function with compact support which equals one one a neighborhood of $K$, then 
we obtain 
\begin{align*}
 \|\Ptj{A(\chiup{f}})\|_0&\leq \|A(\chiup\Ptj{f})\|_0+\|\, [\Ptj,A](\chiup{f})\|_0\\
 &\leq 
 C'\left(\|\chiup \Ptj{f}\|_{\delta}+\|\chiup f\|_{\delta}+{\sum}_{s=1}^{2r_j}\|X_s [X_s,A](\chiup {f})\|_0\right)\\
 &\leq C'' \left(\|\chiup\Ptj{f}\|_{\delta}+\|\chiup f\|_{\delta}+{\sum}_{s=0}^{r_j}\|Z_{s,j}[X_s,A](\chiup{f})\|_0
 \right)\\
 &\leq C'''\left(\|\chiup\Ptj{f}\|_{\delta}+\|\chiup f\|_{\delta}+{\sum}_{s=0}^{r_j}\|\chiup Z_{s,j}{f}\|_\delta
 \right), \quad \forall f\in\Ci(U),
\end{align*}
for suitable positive constants $C',C'',C'''$, uniform with respect to $f$. 
By using similar argument to estimate $\|Z_iAf\|_0$, we obtain that 
\begin{equation*}
 \|A(\chiup{f})\|_{\epsilon}\leq \mathrm{const}\left (\|\chiup f\|_\delta+{\sum}_{j=1}^h\|\chiup{P}_{\tauup_j}f\|_\delta+
 {\sum}_{i=1}^\ell\|\chiup{Z}_if\|_0\right),\quad
 \forall f\in\Ci(U).
\end{equation*}
This shows that for any pair of functions $\chiup_1,\chiup_2\in\Cic(U)$ with $\supp(\chiup_1)\subset
\{\chiup_2>0\}$ we obtain the estimate 
\begin{equation*}
 \|\chiup_1{f}\|_{\epsilon+\delta}\leq 
 \mathrm{const}\left (\|\chiup_2 f\|_\delta+{\sum}_{j=1}^h\|\chiup_2{P}_{\tauup_j}f\|_\delta+
 {\sum}_{i=1}^\ell\|\chiup_2{Z}_if\|_0\right),\quad
 \forall f\in\Ci(U),
\end{equation*}
for some constant $\mathrm{const}=\mathrm{const}(\chiup_1,\chiup_2)\geq{0}$. 
By \cite{F44},
this inequality is valid for all $f 
\in
W^{\delta,2}_{\mathrm{loc}}(U)$ with $\Ptj{f}, Z_if \in 
W^{\delta,2}_{\mathrm{loc}}(U)$, 
where $W^{\delta,2}_{\mathrm{loc}}(U)$ is the space of distributions $\phiup$ in $U$
such that, for all $\chiup\in\Cic(U)$, the product $\chiup\cdot\phiup$ belongs to the
Sobolev space of order $\delta$ and integrability two.
This implies in particular that any $CR$ distribution which is in  $W^{\delta,2}_{\mathrm{loc}}(U)$
belongs in fact to $W^{\delta+\epsilon,2}_{\mathrm{loc}}(U)$, and this implies property $(H)$. 
\end{proof}
\par\smallskip
Let us consider the case where $\Li(\Vf_2;\Hf)$ does not contain all smooth real vector fields.
In this case we have a propagation phenomenon along the leaves of $\Vf_2$. 
 Let $\tauup\in\Kf(M)$, and $X_0,Y_0,X_1,\hdots,X_{2r}$
the vector fields introduced above for a given representation of 
$\tauup=Z_1\otimes\bar{Z}_1+\cdots+Z_r\otimes\bar{Z}_r$. 
As we already noticed,
while $Y_0=\im{\sum}[Z_i,\bar{Z}_i]$ belongs to the Lie subalgebra of $\mathfrak{X}(M)$ 
generated by $X_1,\hdots,X_{2r}$,
the real part $X_0$ of $L_0=X_0+iY_0\in\Zf(M)$ may not belong to $\Vf_1(\tauup)$.
Thus the following result improves \cite[Theorem\,5.2]{HN03}, where only the smaller distribution
$\Vf_1(\tauup)$ was involved.
\begin{thm} \label{thm3.7}
Let  $\Omega^{\mathrm{open}}\subset{M}$ and assume that $\Vf_2$ has constant rank in $\Omega$. 
If
 $f\in\Ot_{\!{M}}(\Omega)$ and $|f|$ attains a maximum at a point $p_0$ of $\Omega$,
 then $f$ is constant on the leaf  through $p_0$ of $\Vf_2$ in $\Omega$. 
\end{thm} 
\begin{proof} On the integral manifold $N$ of $\Vf_2$
through $p_0$ 
 in $\Omega$ we can consider the $\Zf'$-structure defined by the span 
 of the restrictions to $N$ of the elements of $\hat{\Theta}$. 
Indeed, the $CR$
functions on $\Omega$ restrict to $CR$ functions for $\Zf'$ on 
the leaf $N$. By Corollary~\ref{cor3.6} and Theorem~\ref{thm3.3}, 
the $\Zf'$-manifold $N$ has property $(H)$ and therefore
the statement is a consequence of Proposition~\ref{pp2.1}.
\end{proof}

\section{Malgrange's theorem and some applications}\label{s4}
In this section we state the obvious generalization of Malgrange's 
vanishing theorem and its corollary on the extension of $CR$ functions
under momentum conditions,
slightly generalizing results of
\cite{BHN10, CL1999, CL2000} to the case where $M$ has property~$(S\! H)$. 
In this section we require that $M$ is a $CR$ manifold.
\par 
We recall that  the tangential Cauchy-Riemann complex can be defined as the
quotient of the de Rham complex on the powers of the ideal sheaf
(for this presentation we refer to \cite{HN95}): since
$d\If\subset\If$, we have $d\If^{\, a}\subset\If^{\,{a}}$ for all nonnegative integers $a$
and the tangential $CR$-complex 
$(\Qc^{a,*},\bar{\partial}_M)$ 
on $a$-forms is defined by the commutative
diagram \begin{equation} 
\quad
\begin{CD} 0 @>>> \If^{\, a+1} @>>> \If^{\,{a}} @>>> \Qc^{\,{a}} @>>>0\,\\
@. @V{d}VV  @V{d}VV  @V{\bar{\partial}_M}VV \\
0 @>>> \If^{\, a+1} @>>> \If^{\,{a}} @>>> \Qc^{\,{a}} @>>>0,
\end{CD} 
\end{equation} 
where  
$\Qc^{\,{a}}$ is the quotient $\If^{\,{a}}/\If^{\, a+1}.$
In turn $\bar{\partial}_M$ is a degree $1$ derivation for a $\mathbb{Z}$-grading
$\Qc^{\,{a}}={\bigoplus}_{q=0}^n\Qc^{a,q}$, where 
the elements of $\Qc^{a,q}$ are equivalence classes of forms
having representatives in $\If^{\,{a}}\cap\Af^{\C}_{a+q}$. \par 
We denote by $\cE$ the sheaf of germs of smooth complex valued functions on $M$.
The the $\Qc^{a,q}$ are all locally free sheaves of $\cE$-modules, and therefore
we can form the corresponding sheaves and co-sheaves of functions and 
distributions. We will consider the
tangential Cauchy-Riemann complexes $(\cD^{a,*},\bar{\partial}_M)$
on smooth forms with compact support, 
$(\cE^{a,*},\bar{\partial}_M)=(\Qc^{a,*},\bar{\partial}_M)$
on smooth forms with closed support, $(\cDd^{a,*},\bar{\partial}_M)$
on form-distributions, $(\cEd^{a,*},\bar{\partial}_M)$ on form-distributions
with compact support. We use the notation 
$H^q(\mathcal{F}^{a,*}(\Omega),\bar{\partial}_M)$ for the cohomology group
in degree $q$ on ${\Omega^{\mathrm{open}}\subset{M}}$, for $\mathcal{F}$
equal to either one of $\cE,\cD,\cDd,\cEd$.

\begin{prop} \label{prsh4.1} If $M$ has property $(S\!{H})$, 
and either $M$ is compact or has property $(W\! U\! C)$, 
then 
$\bar{\partial}_M:\cEd^{a,0}(M)\longrightarrow\cEd^{a,1}(M)$ 
and $\bar{\partial}_M:\cD^{a,0}(M)\longrightarrow\cD^{a,1}(M)$ 
have closed range
for all integers ${a=0,\hdots,m}$.
\end{prop} \begin{proof}
We can assume that $M$ is connected.
It is convenient to fix a Riemannian metric on $M$, 
and smooth Hermitian products on the complex linear bundles 
$Q^{a,q}M$ 
corresponding
to the sheaves $\Qc^{a,q}$, 
to define $L^2$ and
Sobolev norms, by using the associated smooth regular Borel measure. \par 
By property $(S\!{H})$, we have a subelliptic estimate:
for every $K\Subset{M}$ we can find 
constants $C_K\geq{0},\; c_K>0,\; \epsilon_K>0$
such that 
 \begin{equation}\label{sh4.2}
\|\bar{\partial}_M{u}\|_0^2+C_K\|u\|_0^2\geq c_K\|u\|^2_{\epsilon_K},
\quad\forall u\in\cD^{a,0}(K).
\end{equation}
In a standard way we deduce from \eqref{sh4.2} that 
\begin{equation}\label{sh4.4}
 u\in\cDd^{a,0}(M),\;\; 
\bar{\partial}_Mu\in[\Sob_{\mathrm{loc}}^r]^{a,1}(M)\Longrightarrow
u|_{\ring{K}}\in[\Sob_{\mathrm{loc}}^{r+\epsilon_K}]^{a,1}(\ring{K}),\;\;
\forall K\Subset{M},
\end{equation}
and that for all $K\Subset{M}$ and
real $r$
there are constants $C_{r,K}\geq{0}$, $c_{r.K}>0$ such that
  \begin{align}\label{sh4.3}
\|\bar{\partial}_M{u}\|_r^2& +C_{r,K}\|u\|_r^2\geq 
c_{r,K}\|u\|^2_{r+\epsilon_K},\\
\notag
\quad &\forall u\in\{u\in\cEd^{a,0}(M)\mid 
\bar{\partial}_Mu\in[\Sob^r]^{a,1}(M),\;\supp(u)\subset{K}\}.
\end{align}
This suffices to obtain the thesis when $M$ is compact. \par 

Let us consider the case where $M$ is connected and non-compact.
Let $\{u_\nu\}$ be a sequence in $\cEd^{a,0}(M)$ such that 
all $\bar{\partial}_Mu_\nu$
have support in a fixed compact subset $K$ of $M$ and there is $r\in\R$
such that $\{\bar{\partial}_Mu_\nu\}\subset[\Sob^{r}](M)$,
 $\supp(\bar{\partial}_Mu_\nu)\subset{K}$ for all $\nu$ and
 $\bar{\partial}_Mu_\nu \to f$ in $[\Sob^r]^{a,1}(M)$.
We can assume that $M\setminus{K}$ has no compact connected component. 
Then, since $M$ has property $(W\! U\! C)$, 
it follows that $\supp(u_\nu)\subset{K}$
for all $\nu$, because the $u_\nu|_{M\setminus{K}}$ define elements of
$\Ot_M(M\setminus{K})$ which vanish on a nonempty open subset of
each connected component of 
$M\setminus{K}$, and thus on $M\setminus{K}$. Moreover, this also implies that
\eqref{sh4.3} holds with $C_{r,K}=0$. Then $\{u_\nu\}$
is uniformly bounded in $[\Sob^{r+\epsilon}]^{a,0}(M)$
and hence contains a subsequence which 
weakly converges 
to a solution 
$u\in [\Sob^{r+\epsilon}]^{a,0}(M)$ of $\bar{\partial}_Mu=f$.\par 
The closedness of the image of $\bar{\partial}_M$ in $\cD^{a,1}(M)$
follows from the already proved result for $\cEd^{a,1}(M)$ and the
hypoellipticity of $\bar{\partial}_M$ on $(a,0)$-forms.
\end{proof}
We remind that if $M$ is embedded and has property $(H)$, or is (abstract
and)
essentially pseudoconcave, then has  property $(W\! U\! C)$.\par 
As in \cite{BHN10} one obtains 
\begin{prop}
Assume that $M$ is a connected non-compact $CR$ manifold
of $CR$ dimension $n$ which has properties
$(S\!{H})$ and $(W\! U\! C)$. Then $H^{n}(\cE^{a,*}(M),\bar{\partial}_M)$
and $H^{n}(\cDd^{a,*}(M),\bar{\partial}_M)$ are~$0$ for all $a=0,\hdots,m$. 
\end{prop} \begin{proof}  By Proposition~\ref{prsh4.1},
the sequences \begin{gather*}
\begin{CD} 0 @>>> \cD^{a,0}(M) @>{\bar{\partial}_M}>> \cD^{a,1}(M),\\
0@>>>\cEd^{a,0}(M) @>{\bar{\partial}_M}>> \cEd^{a,1}(M)\,
\end{CD}
\end{gather*} 
are exact and all maps have  closed range.
\par 
Assume that $M$ is oriented. Then we can define duality pairings
between $\cD^{a,q}(M)$ and ${\cDd}^{n+k-a,n-q}(M)$ and between 
${\cEd}^{a,q}(M)$ and $\cE^{n+k-a,n-q}(M)$, extending 
\begin{equation*}
\langle [\alpha]\, ,\, [\beta]\rangle =\int_M \alpha\wedge\beta,
\end{equation*} 
where $\alpha\in\Af_{a+q}(M)\cap\If^{\,{a}}(M)$ has compact support and 
is a representative of $[\alpha]\in\cD^{a,q}(M)$ and
$\beta\in\Af_{m-a-q}(M)\cap\If^{\, {n+k-a}}(M)$ a representative of
 $[\beta]\in\cE^{n+k-a,n-q}(M)$.
Then by duality (see e.g. \cite{ses})
we obtain exact sequences \begin{equation*}
\begin{CD}
0 @<<< \cDd^{n+k-a,n}(M) @<{\bar{\partial}_M}<<\cDd^{n+k-a,n-1}(M), \\
0 @<<< \cE^{n+k-a,n}(M) @<{\bar{\partial}_M}<<\cE^{n+k-a,n-1}(M),
\end{CD}
\end{equation*}
proving the statement in the case where $M$ is orientable. \par 
If $M$ is not orientable, then we can take its oriened 
double covering
$\pi:{\tilde{M}}\to{M}$,
which is a $CR$-bundle with the total space ${\tilde{M}}$ being a
$CR$ manifold of the same
$CR$ dimension and codimension. From the exact sequences 
\begin{equation*}
\begin{CD}
0 @<<< \cDd^{n+k-a,n}({\tilde{M}}) 
@<{\bar{\partial}_{{\tilde{M}}}}<<
\cDd^{n+k-a,n-1}({\tilde{M}}), \\
0 @<<< \cE^{n+k-a,n}({\tilde{M}}) 
@<{\bar{\partial}_{{\tilde{M}}}}<<
\cE^{n+k-a,n-1}({\tilde{M}}),
\end{CD} \end{equation*}
we deduce that  
statement for the nonorientable $M$ by averaging on the fibers. 
\end{proof}
We also obtain the analogue of the Hartogs-type theorem in \cite{CL1999}.
\begin{prop}\label{propo4.3}
Let $\Omega^{\mathrm{open}}\Subset{M}$ 
be relatively compact, orientable, and 
with a piece-wise smooth boundary
$\partial\Omega$. 
If $u_0$ is the restiction to $\partial\Omega$ of an
$(a,0)$-form
$\tilde{u}_0$ of class $\Co^2$ on $M,$ with
$\bar{\partial}\tilde{u}_0$ vanishing to the second order
on $\partial\Omega$, and \begin{equation*}
\int_{\partial\Omega} u_0\wedge\phi=0,\;\;\quad \forall \phi\in\ker(
\bar{\partial}_M:\cE^{n+k-a,n-1}(M')\to \cE^{n+k-a,n}(M')),
\end{equation*}
then there is $u\in\Qc^{a,0}(\Omega)\cap\Co^1(\bar{\Omega})$ with
$\bar{\partial}_Mu=0$ on $\Omega$ and $u=u_0$ on $\partial\Omega$.
\end{prop} \begin{proof} 
We restrain for simplicity to the case $a=0$. The general case can be discussed
in an analogous way. 
If $M$ is not orientable, then the inverse image of $\Omega$ in the double covering
$\pi:\tilde{M}\to{M}$ consists of two disjoint open subsets, both
$CR$-diffeomorphic to $\Omega$. Thus we can and will 
assume that
$M$ is orientable.
\par 
Let $E$ be a discrete set that intersects each 
relatively compact connected
component of $M\setminus\bar{\Omega}$ in a single point and $M'=M\setminus{E}$. 
Note that $M'$ has been chosen in such a way that
no connected component of
$M'\setminus\Omega$ is compact.
\par 
Extending
$\bar{\partial}_M\tilde{u}_0$ by $0$ outside of $\Omega$,
we define a $\bar{\partial}_M$-closed element $f$
of $\cEd^{0,1}(M')$, with support contained in $\bar{\Omega}$. 
The map $\bar{\partial}_M:\cEd^{0,0}(M')
\to\cEd^{0,1}(M')$ has a closed image
by Proposition~\ref{prsh4.1}. Hence to get existence of  
a solution $v\in\cEd^{0,0}(M')$
to ${\bar{\partial}_Mv=f}$ it suffices to prove that $f$ is orthogonal to the
kernel of $\bar{\partial}_M:\cE^{n+k,n-1}(M')\to \cE^{n+k,n}(M')$. 
This is the case
because \begin{equation*}
\int_{M'} f\wedge\phi=\int_{\Omega}(\bar{\partial}_M \tilde{u}_0)\wedge\phi
=\int_{\Omega}(d u_0)\wedge\phi=\int_{\partial\Omega}u_0\phi-
\int_\Omega u_0 d\phi
\end{equation*}
for all $\phi\in\cE^{n+k,n-1}(M')=\Af_{m-1}^{\C}(M')\cap\If^{\,{n+k}}(M')$,
and the last summand in the last term vanishes when $d\phi=\bar{\partial}_M\phi=0$.
A $v\in\cEd^{0,0}(M')$ satisfying $\bar{\partial}_Mv=f$ defines 
a $CR$ function on
$M'\setminus\bar{\Omega}$ that vanishes on some open subset of each connected
component of $M'\setminus\bar{\Omega}$. Thus, for $(W\!{U}\!{C})$ and the
regularity \eqref{sh4.4}, which are consequences of 
$(S\!{H})$, the solution $v$ is $\Co^1$ and has support in
$\bar{\Omega}$. In particular it vanishes on $\partial\Omega$ and therefore
$u=\tilde{u}_0-v$ satisfies the thesis.
\end{proof}
\begin{rmk}
An anaologue of this
\textit{momentum}
theorem  for 
functions on one complex variable states that a  function $u_0$,
defined and continuous
on the boundary of a rectifiable Jordan curve $\mathbf{c}$,
is the boundary value
of a holomorphic function on its enclosed domain  if and
only if ${\int_{\mathbf{c}}u_0(z) p(z)dz=0}$ for all holomorphic polynomials
$p(z)\in\C[z]$. 
\end{rmk}
\section{Hopf lemma and some consequences} \label{sechopf}
In complex analysis  properties of domains are often expressed
in terms of the indices of inertia of the complex Hessian of its exhausting function.
Trying to mimic this aproach in the case of an
 (abstract)
$CR$ manifold $M$, we are confronted with 
the fact that 
 pluri-harmonicity and pluri-sub-harmonicity 
are  well defined only
for sections of a suitable vector bundle $\mathpzc{T}$
(see \cite{AnNa80,43,Sev}),
which 
can be characterized 
in terms of 
$1$-jets when $M$ is embedded. 
We will avoid here this complication, by defining 
the complex Hessian $dd^c\rhoup$ 
as 
an affine subspace of Hermitian symmetric forms on $T^{1,0}M$. 
As we did for the Levi form, we shall consider
its  extension to $H^{1,1}M$, and note that 
it is an invariantly defined function
on 
$\Kf$.
Since 
a $CR$ function canonically determines a  section 
of $\mathpzc{T}$, 
we will succeed in making a very implicit use of the sheaf $\mathpzc{T}$ of
\textit{transversal $1$-jets}  of~\cite{43}.
\par 
In this section we shall 
consider the $P_{\!\tauup}$ 
of \S\ref{s3}, exhibit their relationship to the complex Hessian, and, by using the fact that they are
 degenerate-elliptic operators, draw,
from their 
boundary behavior at  non-characteristic points,
consequences on the properties of $CR$ functions on $M.$
\subsection*{Hopf lemma}
The classical Hopf Lemma also holds for degenerate-elliptic 
operators. 
We have, from 
\cite[Lemma 4.3]{fee2013}:
\begin{prop}\label{prop4.1}
 Let $\Omega$ be a domain in $M$ and $u\in\Co^1(\bar{\Omega},\R)$ satisfy 
 $P_{\!\tauup}u\geq{0}$ on $\Omega$,
 for the operator $P_{\!\tauup}=-X_0+{\sum}_{i=1}^{2r}X_j^2$ of \eqref{2.2}. 
 Assume that $p_0\in\partial\Omega$
 is a $\Co^2$ non-characteristic point of $\partial\Omega$ for $P_{\!\tauup}$ and that 
 there is an open neighborhood $U$ of $p_0$ in $M$ such that 
\begin{align}\label{qq5.1}
 & u(p)<u(p_0),\;\;\forall p\in\Omega\cap{U}.
\end{align}
Then 
\begin{equation} \vspace{-20pt}
du(p_0)\neq{0}. 
\end{equation}  \qed
\end{prop}
The condition that $\partial\Omega$ is non-characteristic at $p_0$ 
for $P_{\!\tauup}$ 
means that, if $\Omega$ is
represented by $\rhoup<0$ near $p_0$, with $\rhoup\in\Co^2$ and $d\rhoup(p_0)\neq{0}$, then
${\sum}_{i=1}^{2r}|X_j\rhoup(p_0)|^2>0$. 
\begin{rmk}
 If $M$ has property $(H)$, then \eqref{qq5.1} is automatically satisfied 
 if $u=|f|$, for $f\in\Ot_{\!{M}}(\Omega)\cap\Co^0(\bar{\Omega})$, when $u(p_0)$ is a local maximum
 and $f$
 is not constant on a half-neighborhood of $p_0$ in~$\Omega$.
\end{rmk}
\begin{cor}
 Let $\Omega$ be an open subset of $M$ and $f\in\Ot_{\!{M}}(\Omega)\cap\Co^2(\bar{\Omega})$,
 $p_0\in\partial\Omega$ 
 with 
\begin{equation}
 |f(p)|<|f(p_0)|,\quad\forall p\in\Omega.
\end{equation}
If $\partial\Omega$ is smooth and $\Theta$-non-characteristic at $p_0$, then 
$d|f|(p_0)\neq{0}$.
\end{cor} 
\begin{proof}
 By the assumption that $\partial\Omega$ is $\Theta$-non-characteristic at $p_0$, the function
 $u=|f|$ is, for some open neighborhood $U$ of $p_0$ in $M$, a solution of $P_{\!\tauup}u\geq{0}$ 
 on $\Omega\cap{U,}$ for an operator $P_{\!\tauup}$  
of the form \eqref{2.2}, obtained from a section $\tauup$
 of $\Kf(U)$, 
 and for which $\partial\Omega$ is non-characteristic at $p_0$. 
\end{proof}
\subsection*{The complex Hessian and the operators {$dd^c$}, {$P_{\!\tauup}$}}
 Denote by $\Af_{\!{1}}$ the sheaf of germs of smooth real valued \mbox{$1$-forms} on $M$,
 by $\Jf_{\!{1}}$ its subsheaf 
 of germs of sections of $H^0M$ and
 by $\If_{\!{1}}$ the degree $1$-homogeneous elements of the ideal sheaf of $M.$
The  elements of $\If_{\!{1}}$ are the germs of  smooth complex valued 
$1$-forms vanishing on~$T^{0,1}M$.\par
 Let $\Omega$ be an open subset of $M$.
 \begin{lem}\label{l5.3}
  If $\alphaup\in\Af_1(\Omega)$, then we can find $\xiup\in\Af_1(\Omega)$
 such that $\alphaup+i\xiup\in\If_{\!{1}}(\Omega)$. 
\end{lem} 
\begin{proof} The sequence 
\begin{equation*} 
\begin{CD}
 0 @>>> \Jf_{\!{1}} @>{i\,\cdot}>> \If_{\!{1}} @>{\re}>> \Af_1 @>>> 0
\end{CD}
\end{equation*}
of fine sheaves is exact, and thus splits on every open subset $\Omega$ of $M$. 
\end{proof}
If $\rhoup$ si a smooth, real valued function on 
$\Omega^{\mathrm{open}}\subset{M}$, by
Lemma~\ref{l5.3}
we can find $\xiup\in\Af_{\!{1}}(\Omega)$ such that $d\rhoup+i\xiup\in\If_{\!{1}}(\Omega)$. 
If $Z\in\Zf(M)$, then $d\rhoup(Z)=-i\xiup(Z)$, 
 $d\rhoup(\bar{Z})=i\xiup(\bar{Z})$, and we obtain 
 \begin{equation*}
 Z\bar{Z}\rhoup=Z(d\rhoup(\bar{Z}))=iZ[\xiup(\bar{Z})],\quad
 \bar{Z}Z\rhoup=\bar{Z}(d\rhoup(Z))=-i\bar{Z}[\xiup(Z)].
\end{equation*}
 Hence 
\begin{equation*}
 [Z\bar{Z}+\bar{Z}Z]\rhoup=i(Z[\xiup(\bar{Z})]-\bar{Z}[\xiup(Z)])=id\xiup(Z,\bar{Z})+i\xiup([Z,\bar{Z}]).
\end{equation*}
We note that $\xiup$ is only defined modulo the addition of a smooth section 
$\etaup\in\!\Jf_{\!{1}}(\Omega)$ 
of the characteristic bundle
$H^0M$, for which 
\begin{equation*}
 id\etaup(Z,\bar{Z})=-i\etaup([Z,\bar{Z}])=\Lf_{\etaup}(Z,\bar{Z}),\quad\forall Z\in\Zf(M).
\end{equation*}
\begin{defn}
 The \emph{complex Hessian of $\rhoup$ at $p_0$} is the affine subspace 
 \begin{equation}
 \Hess_{p_0}(\rhoup)= 
 \{id\xiup_{p_0}\mid \xiup\in\Af_1(\Omega),\; d\rhoup+i\xiup\in\If_1(\Omega)\}.
\end{equation} \end{defn}
\par 
Fix a point $p_0$ where 
$d\rhoup(p_0)\notin{H}^0_{p_0}M$, i.e. 
$\bar{\partial}_M\rhoup(p_0)\neq{0}$,
and consider the level set $N=\{p\in{U}\mid \rhoup(p)=\rhoup(p_0)\}$,
in a neighborhood $U$ of $p_0$ in $\Omega$ where $\bar{\partial}_M\rhoup(p)$
is never $0.$ 
Then $N$ is a smooth real hypersurface and a $CR$-submanifold, of type $(n\! -\! 1, k\! +\! 1)$.
\begin{lem}
 For every $p\in{N}$, we have 
\begin{equation}\label{eqq4.3}
 \{\, \xiup|_N\mid \xiup\in{T}^*_pM\mid d\rhoup(p)+i\xiup\in{T^*_p}^{1,0}M\}\subset{H}^0_pN.
\end{equation}
The left hand side of \eqref{eqq4.3} is an affine hypersurface in $H^0_pN$, with associated
vector space $H^0_pM$.
\end{lem} 
\begin{proof}
 When $Z\in\Zf(U)$ is tangent to $N$, we obtain $0=d\rhoup(Z_p)=-i\xiup(Z_p)$
 and hence $\xiup(\re{Z}_p)=\xiup(\im{Z}_p)=0$ because $\xiup$ is real. This gives $\xiup|_N\in{H}^0_pN$.
 The last statement is a consequence of the previous discussion of the complex Hessian.
\end{proof} 
\begin{defn}
 If $\rhoup$ is a smooth real-valued function defined on a neighborhood $\Omega$ of a point $p_0\in{N}$
 and $\xiup\in\Af_1(\Omega)$ is such that $dr+i\xiup\in\If_1(\Omega)$, then we set 
\begin{equation}\label{E4.6}
 dd^c\rhoup_{p_0}(\tauup):=\tfrac{i}{2}d\xiup(\tauup),\quad\forall \tauup\in\Kf_{p_0}. 
\end{equation}
\end{defn}
Let $\tauup=Z_1\otimes\bar{Z}_1+\cdots+Z_r\otimes{Z}_r\in\Kf(\Omega)$, 
with $\bar{L}_0-L_0={\sum}_{i=1}^r[Z_j,\bar{Z}_j]$
and $L_0,Z_1,\hdots,Z_r\in\Zf(\Omega)$. Let $\xiup\in\Af_1(\Omega)$ 
be such that $d\rhoup+i\xiup\in\If_1(\Omega)$. Then 
\begin{gather*}
d\rhoup(Z_j)+i\xiup(Z_j) =0 \Longrightarrow d\rhoup(\bar{Z}_j)-i\xiup(\bar{Z}_j) =0 \\
\begin{aligned}
\Rightarrow id\xiup(Z_j,\bar{Z}_j)&=i\left(Z_j\xiup(\bar{Z}_j)-\bar{Z}_j\xiup(Z_j)-\xiup([Z_j,\bar{Z}_j])\right)\\
&=Z_jd\rhoup(\bar{Z}_j)+\bar{Z}_jd\rhoup(Z_j)-i\xiup([Z_j,\bar{Z}_j])\\
&=(Z_j\bar{Z}_j+\bar{Z}_jZ_j)\rhoup-i\xiup([Z_j,\bar{Z}_j]).
\end{aligned}
\end{gather*}
We recall that ${\sum}_{i=1}^r[Z_j,\bar{Z}_j]=\bar{L}_0-L_0=2i\im{L}_0$, with $L_0\in\Zf(\Omega)$.
We have
\begin{equation*}
 (d\rhoup+i\xiup)(L_0)=0\Longrightarrow d\rhoup(\re{L}_0)=\xiup(\im{L}_0),\;
 d\rhoup(\im{L}_0)=-\xiup(\re{L}_0)
\end{equation*}
and therefore 
\begin{align*}
 2d{d^c}\rhoup(\tauup)&={\sum}_{i=1}^r{id\xiup}(Z_j,\bar{Z}_j)={\sum}_{i=1}^r(Z_j\bar{Z}_j+\bar{Z}_jZ_j)\rhoup
 -i\xiup\left({\sum}_{i=1}^r[Z_j,\bar{Z}_j]\right)\\
 &={\sum}_{i=1}^r(Z_j\bar{Z}_j+\bar{Z}_jZ_j)\rhoup+2\xiup(\im{L}_0)\\
 &={\sum}_{i=1}^r(Z_j\bar{Z}_j+\bar{Z}_jZ_j)\rhoup-2d\rhoup(\re{L}_0)=2P_{\tauup}\rhoup.
\end{align*}
As a consequence, we obtain: 
\begin{prop}\label{p4.6}
 If $\rhoup$ is a real valued smooth function on the open set $\Omega$ of $M$ and $\tauup\in\Kf(\Omega)$,
 then 
\begin{equation}\label{4,7} \vspace{-22pt}
 dd^c\rhoup(\tauup)=P_{\!\tauup}\rhoup\;\;\text{on $\Omega$.}
\end{equation}
\qed
\end{prop} 
\par\smallskip
\begin{cor}
 The operator $P_{\!\tauup}$ only depends on the section $\tauup$ of $\Kf$ and is independent
 of the choice of the vector fields $Z_1,\hdots,Z_r\in\Zf$ in \eqref{1.4b}. \qed
\end{cor}
\par\smallskip
\begin{cor}\label{cor4.8} 
Let $\Omega^{\mathrm{open}}\subset{M}$. If 
$\rhoup_1,\rhoup_2\in\Ci(\Omega)$ are real valued functions which agree to
the second order at $p_0\in\Omega$, then 
\begin{equation}
dd^c\rhoup_1(\tauup_0)=dd^c\rhoup_2(\tauup_0),\quad\forall \tauup_0\in 
\Kf_{p_0}.\vspace{-22pt}
\end{equation} 
\end{cor} 
\qed\par \bigskip
In particular, $dd^c\rhoup$ is well defined and continuous on the fibers of  
$\Kf$ for 
functions $\rhoup$ which are of class $\Co^2$.

\begin{rmk} There is a subtle distinction between $dd^c\rhoup$, which is the
$(1,1)$-part of an alternate form of degree two, and $\Hess(\rhoup)$, which is
the $(1,1)$-part of a symmetric bilinear form. In fact we 
multiplied by $(i/2)$ the differential in \eqref{E4.6}, and identified the two concepts,
as multiplication by $i$ interchanges skew-Hermitian and Hermtian-symmetric
matrices.
\end{rmk} 
We have:
\begin{lem} Let $\rhoup$ be a smooth real-valued function defined on a neighborhood of $p_0\in{M}$,
with $d\rhoup(p_0)\neq{0}$ and $N=\{p\mid \rhoup(p)=\rhoup(p_0)\}$. 
 The following statements: 
\begin{enumerate}
 \item[$(i)$] every $h\in\Hess_{p_0}(\rhoup)$ has a non-zero positive index of inertia;
 \item[$(ii)$] there exists $\tauup\in\Kf_{p_0}\cap{H}^{1,1}_{p_0}N$ 
 such that $dd^c\rhoup_{p_0}(\tauup)>0$;
 \item[$(iii)$] the restriction of every
 $h\in\Hess_{p_0}(\rhoup)$ to $T^{0,1}_{p_0}N$ 
 has a non-zero positive index of inertia; 
\end{enumerate}
are related by 
\begin{equation*}\qquad\qquad
  (ii)\Longleftrightarrow (iii) \Longrightarrow (i). \qquad\qquad \vspace{-21pt}
\end{equation*}
\end{lem}
\qed\par\bigskip
Set $U^-=\{p\in{U}\mid \rhoup(p)<\rhoup(p_0)\}$. 
\begin{defn}
\label{df4.3}
We set 
\begin{equation}
{H^0_{M,p_0}}(U^-)=
 {\bigcup}_{\lambda>0} \{\xiup|_N\mid \xiup\in{T}^*_{p_0}\partial{U}^-\mid 
 \lambda d\rhoup(p_0)+i\xiup\in{T^*}^{1,0}_pM\}.
\end{equation}
\end{defn}

This is an open half-space in $H^0_pN$. 
Note that 
${H^0_{M,p_0}}(U^-)$
does not depend on the choice of the defining function $\rhoup$.
\subsection*{Real parts of \textsl{CR} functions} 
In this subsection, we try to better explain the meaning of $dd^c$ by
defining a differential operator $\dc$ which associates to a real smoot
function a real one-form. Its definition depends on the choice of a 
\textit{$CR$ gauge} $\lambdaup$ on $M$, but 
$[\dc]$'s corresponding to different choices of
$\lambdaup$  differ by a differential operator
with values in $\Jf$, so that all the $d\dc$ agree with our $dd^c$ on
$\Kf$.
\par 

A $CR$ function (or distribution) $f$ is a solution to the equation 
$du\in\If_{\!{1}}$. 
In this subsection we study the characterization of the real parts of $CR$ functions.\par
\begin{lem} Let $\Omega$ be open in $M$.
 If $M$ is minimal, then a real valued $f\in\Ot_{\!{M}}(\Omega)$ is locally constant. 
\end{lem} 
\begin{proof}
 A real valued $f\in\Ot_{\! M}(\Omega)$ satisfies $Xf=0$ for all $X\in\Gamma(M,HM)$
 and therefore is constant on the Sussmann leaves of $\Gamma(M,HM)$.
\end{proof}
We have an exact sequence of fine sheaves (the superscript $\C$ means forms with complex valued
coefficients)
\begin{equation}\label{rp1}
 \begin{CD}
 0 @>>> \Jfc 
 @>{\alpha\to(\alpha,-\alpha)}>> 
 \If_1\oplus\Ifc @>{(\alpha,\beta)\to{\alpha+\beta}}>> 
 \Afc@>>>0.
\end{CD}
 \end{equation}\par
 In \cite[\S{2A}]{43} the notion of a 
 \textit{balanced real $CR$-gauge} was introduced. It was shown that it is possible to define
 a smooth morphism \begin{equation}
 \lambdaup:\C{T}M\longrightarrow{T^*}^{1,0}M\end{equation}  of $\C$-linear bundles which
 defines 
a special 
splitting of \eqref{rp1}: with 
\begin{equation}
 \bar{\,\lambdaup}:\C{T}M\ni\alpha\longrightarrow \overline{\lambdaup(\bar{\alpha})}\in{T^*}^{0,1}M,
\end{equation}
we have 
\begin{align}
 \alpha=\lambdaup(\alpha)+\bar{\,\lambdaup}(\alpha),\quad\forall \alpha\in
 \Af_1^{\,\C}, \\
 \lambdaup(\alpha)=\bar{\,\lambdaup}(\alpha)=\tfrac{1}{2}\alpha,\;\;\forall\alpha\in
 \Jf_{\!{1}}^{\,\C}.
\end{align}
Note that 
\begin{equation*}
 \bar{\,\lambdaup}(\If_1)\subset\Jf_1,\;\; \lambdaup(\bar{\If\,\,\,\,}_{\!\!\!\!\!{1}})\subset\Jf_1,\; \lambdaup\circ
 \bar{\,\lambdaup}=\bar{\,\lambdaup}\circ\lambdaup.
\end{equation*} 
\par
Explicitly,
the splitting of \eqref{rp1} is provided by 
\begin{equation*} 
\begin{CD}
 0 @>>>
 \Af_1^{\,\C} @>{\alpha\to(\lambdaup(\alpha),\bar{\,\lambdaup}(\alpha))}>>\If_1\oplus
 \Ifc @>{(\alpha,\beta)\to\bar{\,\lambdaup}(\alpha)-\lambdaup(\beta)}>>
 \Jf_1^{\,\C} @>>> 0.
\end{CD}
\end{equation*}
Furthermore, we get 
\begin{gather*}
 \Afc=\ker\bar{\,\lambdaup}\oplus\Jfc\oplus\ker\lambdaup,\quad
 \If_{\!{1}}=\ker\bar{\,\lambdaup}\oplus\Jfc,\quad
 \Ifc=\Jfc\oplus\ker\lambdaup,\\
 \lambdaup(\alpha)=\alpha,\;\forall \alpha\in\ker\bar{\,\lambdaup},\quad
 \bar{\,\lambdaup}(\alpha)=\alpha,\;\forall\alpha\in\ker\lambdaup,\quad
 \lambdaup(\alpha)=\bar{\,\lambdaup}(\alpha)=\tfrac{1}{2}\alpha,\;\forall\alpha\in\Jfc.
\end{gather*}
\par
Let us introduce the first order linear partial differential operator 
\begin{equation}
 \dc\,{f}=\tfrac{1}{i}(\lambdaup(df)-\bar{\,\lambdaup}(df)),\quad \forall f\in\Ci(M).
\end{equation}
We note that $\dc$ is \textit{real}: this means that $\dc\,{u}$ is a real valued form when $u$ is a
real valued function. 
Indeed,
for a real valued $u\in\Ci(M)$, we have \par
\centerline{$\dc\,{u}=2\im\lambdaup(du)=-2\im(\bar{\,\lambdaup}(du))$.}
\begin{lem}
 We have $d\dc\,{u}\in\!\!\Jf_{\!{2}}$ for every $u\in\Af_0$.
\end{lem}\begin{proof}
For any germ of 
real valued smooth function $u$ the differential $d\dc\,{u}$ is real and we have 
\begin{align*}
 i\,d\dc\, {u}=d(\lambdaup(du)-\bar{\,\lambdaup}(du))&=d(2\lambdaup(du)-du)=2d\,\lambdaup(du)\in\If_{\!{2}},\\
 &=d(du-2\bar{\,\lambdaup}(du)=-2d\bar{\,\lambdaup}(du)\in\,\bar{\!\!\!\If}_{\!{2}},
\end{align*}
so that $d\dc\, {u}\in\If_{\!{2}}\cap\,\bar{\!\!\!\If}_{\!{2}}\cap\Af_2=\Jf_{\!{2}}$. \end{proof}
\begin{prop} \label{lm:5.2}
 Let $\Omega$ be a simply connected open set in $M$. 
 A necessary and sufficient condition for 
 a real valued $u\in\Ci(\Omega)$ 
 to be the real part of an ${f\in\Ot_{\!{M}}(\Omega)}$ is that 
 there exists a section $\xiup\in\Jf_{\!{1}}(\Omega)$ such that
\begin{equation}\label{eq:5.6}
 \df\,[\dc\,{u}+\xiup]=0\;\;\text{on $\Omega$}.
\end{equation}
If $M$ is minimal, then $\xiup$ is uniquely determined.
 \end{prop} 
\begin{proof} Assume that \eqref{eq:5.6} is satisfied by some $\xiup\in\Jf_{\!{1}}(\Omega)$.
Then $\dc\,{u}+\xiup=dv$ for some real valued $v\in\Ci(\Omega)$ and, with $f=u+iv$ we obtain
\begin{align*}
\lambdaup(du)-\bar{\,\lambdaup}(du)=i[\lambdaup(dv)+\bar{\,\lambdaup}(dv)-\xiup]
 \Longrightarrow \bar{\,\lambdaup}(df)=\lambdaup(du-idv)-i\xiup\in\Jfc(\Omega) \\
 \Longrightarrow df\in\If_{\!{1}}(\Omega) \Longleftrightarrow f\in\Ot_{\!{M}}(\Omega).\quad
\end{align*} \par 
Assume vice versa that $f=u+iv\in\Ot_{\!{M}}(\Omega)$, with $u$ and $v$ real valued smooth functions.
Write $df=du+idv=\alpha+\zetaup$, with $\alpha\in\If_{\!{1}}(\Omega)$, $\zetaup\in\Jfc(\Omega)$, and
$\bar{\,\lambdaup}(\alpha)=0$. From 
\begin{equation*}
 \bar{\,\lambdaup}(du)+i\bar{\,\lambdaup}(dv)=\tfrac{1}{2}\zetaup \Longrightarrow
 \lambdaup(du)-i\lambdaup(dv)=\tfrac{1}{2}\bar{\zetaup}\,,
\end{equation*}
we obtain 
\begin{equation*}
 i\dc\, {u}=\lambdaup(du)-\bar{\,\lambdaup}(du)=
 i \lambdaup(dv)+\tfrac{1}{2}\bar{\zetaup} +i\bar{\lambda}(dv)=
 i\, dv-\tfrac{1}{2}({\zetaup}-\bar{\zetaup})
\end{equation*}
This is \eqref{eq:5.6} with $\xiup=(i/2)({\zetaup}-\bar{\zetaup})$. 
\par To complete the proof, we note that, if $\xiup\in\Jf_{\!{1}}(\Omega)$ and $d\xiup=0$, then
$\xiup=d\phi$ for some real valued function
$\phi\in\Ci(\Omega)$. If $\xiup_{p_0}\neq{0}$ for some $p_0\in\Omega$,
then $\{\phi(p)=\phi(p_0)\}$ defines a germ of smooth hypersurface through $p_0$ which is
tangent at each point to the distribution $HM$, contradicting the minimality assumption. 
\end{proof}
The \textit{Aeppli complex} for \textit{pluri-harmonic} functions on the $CR$ manifold $M$ is 
\begin{equation*} 
\begin{CD}
 0 @>>> \Af_0\oplus\Jf_1@>{(u,\xiup)\to d\dc\,{u}+d\xiup}>> \Jf_{\!{2}} @>{d}>> \Jf_{\!{3}}@>{d}>> \\
 @.@. \cdots @>{d}>> \Jf_{\!{m-1}} @>{d}>> \Jf_{m}@>>> 0.
\end{CD}
   \end{equation*}
  \par
We note that $\Jf_{\!{1}}=0$ if $M$ is a complex manifold (we reduce to the classical case)
and $\Jf_{\!{q}}=\Af_{\!{q}}$ for $q>0$ if $M$ is totally real.
In general, the terms of degree $\geq{k\! + \! 2}$ make a sub-complex of the de Rham complex.
\subsection*{Peak points of $CR$ functions and pseudo-convexity
at the boundary} A non-characteristic point of the
boundary of  a domain, where the modulus a  $CR$ function attains a local maximum,
is \textit{pseudo-convex}, in a sense that will be explained below.   
\begin{lem}\label{lm5.7}
 Let $\Omega^{\mathrm{open}}\subset{M}$ 
 and assume there is 
 $f\in{\Ot_{\!{M}}(\Omega)\cap\Co^2(\bar{\Omega})}$
 such that $|f|$ attains a local isolated maximum value at $p_0\in\partial\Omega$.  
 If
  if $\partial\Omega$ is smooth, non-characteristic at $p_0$ and, moreover,
 $d|f(p_0)|\neq{0}$, then there is a non-zero
 $\xiup\in{H}^0_{M,p_0}(\Omega)$ with $\Lf^{\partial\Omega}_{\xiup}\geq{0}$.
\end{lem} 
\begin{proof} 
 Let $U$ be an open neighborhood of $p_0$ in $M$,
and $\rhoup\in\Ci(U,\R)$ a defining function for $\Omega$ near $p_0$, with 
$U^-=\Omega\cap{U}=\{p\in{U}\mid \rhoup(p)<0\}$,
and  $d\rhoup(p)\neq{0}$ for all $p\in{U}$.\par 
 We can assume that $f(p_0)=|f(p_0)|>0$ and exploit the fact that 
 the restriction of $u=\re{f}$ to $\partial\Omega$ takes a maximum value at $p_0$. 
 Since $d_{\partial\Omega}u(p_0)=0$, the real Hessian of $u$ on $\partial\Omega$ is
 well defined at $p_0$, with 
\begin{equation*}
 \mathrm{hess}(u)(X_{p_0},Y_{p_0})=(XYu)(p_0),\quad \forall X,Y\in\mathfrak{X}(\partial\Omega),
\end{equation*}
and $\mathrm{hess}(u)(p_0)\leq{0}$ by the assumption that 
the restriction of $u$ to $\partial\Omega$ has a local maximum at $p_0$.
In particular, it follows that 
\begin{equation*}
(Z\bar{Z}u)(p_0)= (\bar{Z}{Z}u)(p_0)\leq{0},\quad\forall Z\in\Zf(\partial\Omega).
\end{equation*}
Let $v=\im{f}$. Then $df=du+idv$, and the condition that $d_{\partial\Omega}u(p_0)=0$ 
implies that $(Zv)(p_0)=0$ for all $Z\in\Zf(\partial\Omega)$ and thus
$\xiup=dv(p_0)\in{H}^0\partial\Omega$. Moreover, 
\begin{align}\label{eq5.4}
 (Zu)(p)=-i(Zv)(p),\;\; (\bar{Z}u)(p)=i(\bar{Z}v)(p),
 \quad \forall Z\in\Zf(\partial\Omega),\quad\forall p\in\partial\Omega.
\end{align}
Hence 
\begin{align*}
 2Z\bar{Z}u(p_0)&=(Z\bar{Z}+\bar{Z}Z)u(p_0)=i(Z\bar{Z}-\bar{Z}Z)v(p_0)=
 i\xiup(p_0)([Z,\bar{Z}])
\end{align*}
and thus the condition on the \textit{real} Hessian of $u$ implies that
 $\Lf_{\xiup}^{\partial\Omega}\geq{0}$. We note that $du(p_0)$ is different from $0$ and proportional to
 $d\rhoup(p_0)$. 
 Indeed, near $p_0$ we have 
\begin{equation*}
 |f|=u\sqrt{1+(v^2/u^2)}\simeq u\big(1+\tfrac{1}{2}(v^2/u^2)\big)=u+0(2),
\end{equation*}
since $v(p_0)=0$. Thus $d|f|(p_0)=du(p_0)\neq{0}$.
\par
 By the assumption that $\partial\Omega$ is non-characteristic at $p_0$, we have
 that ${d}u(p_0)$ is non-zero and equal to $\lambda{d}\rhoup(p_0)$
 for some $\lambda>0$: therefore 
 $\xiup=dv(p_0)\in{H}^0_{M,p_0}(\Omega)$ and 
 this proves our claim. 
\end{proof}
\begin{prop}\label{lm5.2}
 Let $\Omega$ be an open subset of $M$, and assume that there is a $CR$ function 
 $f\in\Ot_{\!{M}}(\Omega)\cap\Co^2(\bar{\Omega})$ and a point $p_0\in\partial\Omega$ such that: 
\begin{align}
 \tag{$a$}  & |f(p_0)|>|f(p)|,\;\;\forall p\in\Omega,\\
 \tag{$b$} & \partial\Omega \;\;\text{is 
 $\Theta$-non-characteristic at $p_0$.  }
 \end{align}
 \par 
Then we can find $0\neq{\xiup}\in{H}^0_{M,p_0}(\Omega)$ with  
$\Lf_\xiup^{\partial\Omega}\geq{0}$.
\end{prop}
[For the meaning of \textit{non-characteristic} see \eqref{equnosei}.]
\begin{proof} 
To apply Lemma~\ref{lm5.7} we need to check that $d|f|(p_0)\neq{0}$. 
 By the assumption that $\partial\Omega$ is $\Theta$-non-characteristic at $p_0$,
 there is an open neighborhood $U$ of $p_0$ in $M$ and
 $\tauup\in\Kf(U)$ such that $\partial\Omega$ is non-characteristic for $P_{\!\tauup}$
 at $p_0$. 
 Since $P_{\!\tauup}|f|\geq{0}$, 
by the Hopf lemma,   $d|f|(p_0)|\neq{0}$ and therefore $du(p_0)$ 
is a positive multiple of $d\rhoup(p_0)$. Then 
$\xiup=d\im\!{f}(p_0)\in{H}^0_{M,p_0}(\Omega)$ and we obtain the 
statement. 
\end{proof} 
For characteristic peak points in the boundary of $\Omega$ we have:
\begin{lem}\label{lm5.3}
Let $\Omega$ be an open subset of $M$, and assume that there is a $CR$ function 
 $f=u+iv\in\Ot_{\!{M}}(\Omega)\cap\Co^2(\bar{\Omega})$, with $u$ and $v$ real valued,
 and  $p_0\in\partial\Omega$ such that: 
\begin{align}
 \tag{$a$}  & v(p_0)=0,\;\; du(p_0)\in{H}_{p_0}^0N,\;\; u(p_0)>u(p),\;\;\forall p\in\Omega,\\
 \tag{$b$} & 0\neq \xiup=dv(p_0). 
\end{align} \par 
Then   $\xiup\in{H}^0_{p_0}M$ and $\Lf_\xiup\geq{0}$. 
\end{lem} 
\begin{proof}
 Set $\etaup=du(p_0)$. Then
 $\xiup=dv(p_0)\in{H}^0_{p_0}M$, because $df(p_0)=\etaup+i\xiup$ is zero on
 $\Zf(M)$, and hence $\xiup$, vanishing  on $\Zf(M)$ and being real, belongs to $H^0_{p_0}M$.
 The conclusion follows by the argument of Lemma~\ref{lm5.7}, taking into account that this
 time all vectors in $T^{0,1}_{p_0}M$ are tangent to $\partial\Omega$ and that \eqref{eq5.4}
 is valid for $Z\in\Zf(M)$ at all points where $f$ is defined and $\Co^1$.
\end{proof} 
Proposition~\ref{lm5.2} 
suggest to introduce some notions of
convexity/concavity for  boundary points of a domain in $M$.
Let $\Omega$ be a domain in $M$, 
 $p_0\in\partial\Omega$ a smooth point of $\partial\Omega$, and
 $\rhoup$ a defining function for $\Omega$ near $p_0$. 
\begin{defn}\label{d4.4}
We say that $\Omega$ is at $p_0$ 
\begin{itemize}
 \item strongly-$1$-concave if there is  
 $\tauup\in\Kf\cap{H}^{1,1}_{p_0}\partial\Omega$
 such that 
 $dd^c\rhoup_{p_0}(\tauup)<0$;
 \item strongly-$1$-convex if there is 
  $\tauup\in\Kf\cap{H}^{1,1}_{p_0}\partial\Omega$ such that 
 $dd^c\rhoup_{p_0}(\tauup)>0$.
\end{itemize}
\end{defn}
Points where the boundary is strictly $1$-concave cannot be peak points for
the modulus of $CR$ functions. 
\begin{prop} Assume that $M$ has property $(H)$.
 Let $\Omega$ be a relatively compact open domain in $M$ and $N\subset\partial\Omega$ 
a smooth part of $\partial\Omega$ consisting of points where $\partial\Omega$ is smooth,
 $\Theta$-non-characteristic, 
 and strongly-$1$-concave. Then 
 \begin{equation}
 |u(p)|<\sup_{q\in\partial\Omega\setminus{N}}|u(q)|,\quad\forall p\in\Omega\cup{N},
 \end{equation}
for every non constant $u\in\Ot_{\!{M}}(\Omega)\cap\Co^2(\bar{\Omega})$.
\end{prop} 
\begin{proof}
 Since $M$ has property $(H)$, by Proposition~\ref{pp2.1} we have 
 $|f(p)|<{\max}_{\partial\Omega}|f|$, for all $p\in\Omega$
and all non constant $f\in\Ot_{\!{M}}(\Omega)$. The statement then follows from Proposition~\ref{lm5.2},
because $|f|$ cannot have a maximum on $N$.
\end{proof}

\subsection*{{$1$}-convexity/concavity at the boundary and the 
vector-valued Levi form} 
Let $\Omega^{\mathrm{open}}\subset{M}$ have piece-wise smooth boundary and
denote  by $N$ the $CR$ submanifold of type $(n\! -\!1, k\! + \! 1)$ 
of $M$ consisting of the smooth
non-characteristic points of $\partial\Omega$. 
The quotient $(TN\cap{HM})/HN\subset{TN}/HN$ is
a real line 
bundle on $N$.\par 
The partial complex structure $J_M:HM\to{HM}$ restricts to the partial complex structure on $HN$ and
the tangent vectors $v$ in $(HM\cap{TN})\setminus{HN}$ are characterized by the fact that
$J_M(v)\notin{TN}$. Fix a point $p_0\in{N}$ and a defining function $\rhoup$ of $\Omega$ on
a neighborhood $U$ of $p_0$ in $N$, so that $0\neq{d}\rhoup(p_0)$ is an outer conormal to $\Omega$
at $p_0$. The elements $\xiup_0\in{H}^0_{M,p_0}\Omega
$ are defined, modulo
multiplication by a positive scalar,  
by the
condition that $d\rhoup(p_0)+i\xiup_0\in{T^*}_{p_0}^{1,0}M$. Since $v+iJ_Mv\in{T}^{0,1}_{p_0}M$, 
we have 
\begin{align*}
0= \langle(d\rhoup(p_0)+i\xiup_0),(v+iJ_Mv)\rangle=i\langle d\rhoup(p_0), J_Mv\rangle +i\langle
\xiup_0,v\rangle -\langle \xiup_0,J_Mv\rangle\qquad 
\\
\Longrightarrow 
\langle \xiup_0,J_Mv\rangle=0,\quad 
\langle\xiup_0,v\rangle=-\langle{d}\rhoup(p_0),J_Mv\rangle.
\end{align*}
The restriction $\xiup_0|_N$ is an element of $H^0_{p_0}N$, with $\langle\xiup_0,v\rangle\neq{0}$
if $p_0$ is non-characteris\-tic. Therefore we have shown:
\begin{lem}\label{lem5.18}
Let $v=J_M{w}_{p_0}$  for an outer normal vector in $p_0\in{N}\subset\partial\Omega$ to $\Omega$,
with $v\in{H}_{p_0}M$. If $[v]$ belongs to the range of the vector-valued Levi form
$\Lf^N$, then $\Omega$ is strongly-$1$-convex at $p_0$.\par
Vice versa, if $\Omega$ is strongly-$1$-convex at $p_0$, then $[v]$ belongs to the range
of the vector valued Levi form.\qed
\end{lem}
As usual, we used $[v]$ to denote the image of $v$ in the quotient $TN/HN$.\par
%
A similar statement holds for strong-$1$-concavity.
\section{Convex cones of Hermitian forms}\label{cv}
In a $CR$ manifold of arbitrary $CR$-codimension, the \textit{scalar} Levi forms associate
to each point a linear space of Hermitian symmetric quadratic forms. 
Different notions of pseudo-concavity in \cite{AHNP08a,HN00,HN03} originate from
the observation 
 that the polar of a subspace of forms with
 positive Witt index contains positive definite 
tensors. As we 
showed in \S\ref{sechopf},
the 
analogue on a $CR$ manifold $M$ of the 
complex Hessian of a smooth real function
yields
an \textit{affine} subspace of Hermitian symmetric forms.
Therefore 
it was 
natural to
associate to a non-characteristic point of the boundary of a domain 
in $M$ 
an open half-space of Hermitian symmetric forms.
In this section we describe
some properties of duals of convex cones 
of Hermitian symmetric forms, to better understand 
the notions of pseudo-concavity 
that are relevant to discuss the extensions of some facts of  analysis 
in several complex variables
to
the case of $CR$ manifolds.
\subsection*{Convexity in Euclidean spaces} (cf. \cite{Klee,Rk})
Let us recall some notions of convex analysis. Let $V$ be an $n$-dimensional 
Euclidean real
vector space. A nonempty subset $C$ of $V$ is a convex cone (with vertex $0$) if 
\begin{equation*}
 v_1,v_2\in{C}, \; t_1>0,\; t_2\geq{0} \Longrightarrow t_1v_1+t_2v_2\in{C}.
\end{equation*}
The \textit{dual cone} of $C$ is 
\begin{equation*}
 C^*=\{\xiup\in{V}\mid ({v}|\xiup)\geq{0},\;\forall v\in{C}\}.
\end{equation*}
By the Hahn-Banach theorem, one easily obtains: 
\begin{lem}\label{lemconv1}
 For any nonempty convex cone $C$ in $V$ we have $C^{**}=\bar{C}$.
\end{lem} 
\begin{proof}
 If $w\notin\bar{C}$, then, by the Hahn-Banach separation theorem we can find $\xiup\in{V}$ 
 such that ${\inf}_{v\in{C}}(v|\xiup)>(w|\xiup)$. Since $C$ is a cone, this implies that
 $(v|\xiup)\geq{0}$ for all $v\in{C}$, i.e. $\xiup\in{C}^*,$ 
 and then $(w|\xiup)<0$ shows that
$w\notin{C}^{**}$. 
 This proves 
 that $C^{**}\subset\bar{C}$. 
 The opposite inclusion trivially follows from the defintion. 
\end{proof}
We call \emph{salient} a convex cone which does not contain any real line:
this means that if $0\neq{v}\in{C}$,
then ${-v}\notin{C}$. By Lemma~\ref{lemconv1} we have 
\begin{lem}\label{lemconv2}
 A nonempty closed convex cone $C$ is 
 salient %
 if and only if $C^*$ has a nonempty interior.
\end{lem} 
\begin{proof}
 If $C$ contains a vector subspace $W$, then $C^{*}$ is contained in the orthogonal
 $W^*=W^\perp$, 
 which is a proper linear
 subspace of $V$ and therefore $C^*$ has an empty interior.
 Vice versa, if $C^*$ has an empty interior, 
 then its linear span $U$ is a proper linear subspace of
 $V$ and $W=U^*=U^\perp$ is a linear subspace of $V$ 
 of positive dimension contained in $\bar{C}=C$.
\end{proof}
\begin{lem}\label{lemconv3}
Let $C$ be a 
salient
closed convex cone and $W$ a linear subspace of $V$ 
with $W\cap{C}=\{0\}$.
Then we can find a hyperplane $W'$ with $W\subset{W}'$ and $W'\cap{C}=\{0\}$.  
\end{lem} 
\begin{proof} For each  
$v\in{V}$ we write $v=v'+v''$ for its decomposition into the sum of its component
$v'\in{W}$
and its component $v''\in{W}^\perp$. 
We claim that
the orthogonal projection $C''$ of $C$ into $W^\perp$ is still a closed salient cone.  
Closedness follows by the fact that $\|v'\|\leq{C}\|v''\|$ for some $C>0$ for all $v\in{C}$. 
To prove that $C''$ is salient,
we argue by contradiction. Assume that $C''$ contains
two opposite nonzero vectors $\pm{w}''$. Then there are $w'_+,w'_-\in{W}$ such that 
$w'_++w'',w'_--w''\in{C}$. The sum of these two nonzero vectors is nonzero by the assumption that
$C$ is salient, but $$0\neq (w'_++w'')+(w'_--w'')=(w'_++w'_-)\in{C}\cap{W}$$ 
yields a contradiction.
%
%
%
\par 
By Lemma~\ref{lemconv2}, the interior of the dual cone of 
$C''$ in $W^\perp$ is nonempty.
This means that there is a $\xiup\in{W}^\perp$ with $(v''|\xiup)
>0$ 
for all $v''\in{C}''$ and hence $(\xiup|v)>0$ for all $v\in{C}$, since $C\subset{C}''+W$.
\end{proof}
A closed convex cone $C$ with $\ring{C}^*=\emptyset$ contains a linear subspace $E_C$
of $V$ and is called a \textit{wedge} with \textit{edge} $E_C$. Lemma~\ref{lemconv3} 
generalizes to the case of closed wedges.
\begin{lem}\label{lemconv4}
 If $C$ is a closed wedge with edge $E_C$ and $W$ a linear subspace of $V$ with
 $W\cap{C}\subset{E}_C$, then there is a hyperplane $W'$ with $W\subset{W}'$ and
 $W'\cap{C}=E_C$.
\end{lem} 
\begin{proof}
 $C$ contains all affine subspaces $v+E_C$, for $v\in{C}$. If $\pi:V\to{V}/E_C$ is the projection
 into the quotient, then $\pi(C)$ is a pointed cone and $\pi(W)\cap\pi(C)=\{0\}$. 
 By Lemma~\ref{lemconv3} there is
 a hyperplane $H$ in $V/W$ with $\pi(W)\subset{H}$
 and $H\cap\pi(C)=\{0\}$. Then $W'=\pi^{-1}(H)$ is a hyperplane in $V$ which contains $W$
 and has ${C\cap{W}'=E_C}$.
\end{proof}
\subsection*{Convex cones in the space of Hermitian symmetric forms}
Let us denote by $\Pf_n$ the $n^2$-dimensional real vector space of $n\!\times\!{n}$ 
Hermitian symmetric forms on $\C^n$. It is a Euclidean space with the scalar
product $(h_1|h_2)={\sum}_{i,j=1}^n h_1(e_i,e_j)h_2(e_j,e_i)$, where
$e_1,\hdots,e_n$ is any basis of $\C^n$. It will be convenient however
to avoid fixing any specific scalar product on $\Pf_n$, and formulate
our statements in 
a more invariant way, involving 
the
dual $\Pf_n'$ of $\Pf_n$. 
It consists of the
Hermitian symmetric covariant tensors, that we write as sums $\pm{v}_1\otimes\bar{v}_1\pm
\cdots \pm  v_r\otimes\bar{v}_r$, for $v_1,\hdots,v_r\in\C^n$.
The identification of $\Pf_n$ with $\Pf_n'$ provided by the choice of
a scalar product on $\Pf_n$ allows us to apply the 
previous
results of convex analysis
in this slightly different formulation. \par
A matrix corresponding to a
Hermitian symmetric form $h$ has real eigenvalues. The number of  positive (resp. negative)
eigenvalues is called its \textit{positive} (resp. \textit{negative}) \textit{index of inertia}, the 
smallest of the two its \emph{Witt index}, the sum of the two its \emph{rank}. \par 
Set $\bar{\Pf}_n^+=\{h\geq{0}\}$ and $\Pf_n^+=\bar{\Pf}_n^+\setminus\{0\}$, $\ring{\Pf}_n^+=\{h>0\}$,
and, likewise, $\bar{\Pf}_n^-=\{h\leq{0}\}$ and $\Pf_n^-=\bar{\Pf}_n^-\setminus\{0\}$,
$\ring{\Pf}_n^-=\{h<0\}$. 
We shall use
the simple 
\begin{lem}\label{lem6.5}
\begin{gather*}
 [\bar{\Pf}_n^+]^*=[\ring{\Pf}_n^+]^*={\bigcup}_r\{v_1\otimes\bar{v}_1+\cdots+v_r\otimes\bar{v}_r\mid
 v_1,\hdots,v_r\in\C^n\},\\
 \{\psiup\in\Pf_n'\mid \psiup(h)>0,\;\forall h\in\Pf_n^+\}=\{v_1\otimes\bar{v}_1+\cdots+v_n\otimes\bar{v}_n\mid
\langle v_1,\hdots,v_n\rangle=\C^n\},\\
\{\psiup\in\Pf_n'\mid \psiup(h)>0,\;\forall h\in\ring{\Pf}_n^+\}=
\{v_1\otimes\bar{v}_1+\cdots+v_r\otimes\bar{v}_r\mid r>0,\;
\langle v_1,\hdots,v_n
\rangle
=\C^n\}.
\end{gather*}
\end{lem}
\begin{prop}\label{prc1}
 Let $\Wf$ be a 
 convex closed cone, with vertex in $0$, in $\Pf_n$. 
 Assume that  every nonzero
 element of $\Wf$ has a non-zero positive index of inertia. Then there is a  basis
 $e_1,\hdots,e_n$ of $\C^n$ such that 
\begin{equation}
 {\sum}_{i=1}^n h(e_i,e_i)\geq{0},\quad\forall h\in\Wf.
\end{equation}
\end{prop} 
\begin{proof} Both $\Wf$ and 
${\Wf}^+=\{h_1+h_2\mid h_1\in\Wf,\; h_2\geq{0}\}$ are
 proper closed
convex 
cones 
in $\Pf_n$. Since $\Wf^+$ does not contain any negative semidefinite nonzero form,
its edge has empty intersection with
$\Pf_n^+=\{h\geq{0},\; h\neq{0}\}$. By Lemma~\ref{lemconv4}
we can find a $\psiup\in\Pf_n'$ such that 
\begin{equation*}
 \psiup(h)\geq{0},\;\forall h\in{\Wf}^+ \quad\text{and}\quad \Wf^+\cap\{\psiup=0\}=E_{\Wf^+}.
\end{equation*}
In particular, $\psiup(h)>0$ for $h\in\Pf_n^+$ and hence,
by Lemma~\ref{lem6.5},
 $\psiup$ is of the form
$\psiup(h)={\sum}_{i=1}^n{h}(e_i,e_i)$ for a basis $e_1,\hdots,e_n$ of $V$.
\end{proof}
We obtain, as a corollary, the result of \cite[Lemma 2.4]{HN00}, which motivated the definition
of \textit{essential pseudo-concavity}.
\begin{cor}
 If $\Wf$ is a linear subspace of $\Pf_n$ such that each nonzero element of $\Wf$ has a positive
 Witt index, then there exists a basis $e_1,\hdots,e_n$ of $\C^n$ such that 
\begin{equation*} \qquad\qquad\qquad\qquad\qquad\qquad
 {\sum}_{i=1}^n h(e_i,e_i)=0. \qquad\qquad\qquad\qquad\qquad\qquad\qed
\end{equation*}
\end{cor}

\begin{prop}\label{pp6.1}
 Let $\Wf$ be a relatively open convex cone
with vertex at $0$ 
 of $\Pf_n$, and such that every element $h$ of $\Wf$ has a non-zero positive index of inertia. Then
 the elements of $\bar{\Pf}^-_n$ which are contained in $\overline{\Wf}$ are all degenerate. 
 \par
 All the elements of maximal rank in $\overline{\Wf}\cap\bar{\Pf}_n^-$
 have the same kernel, which has a positive dimension $r$ and a basis
 $e_1,\hdots,e_r$ such that 
\begin{equation}\label{q5.1}
 {\sum}_{i=1}^rh(e_i,e_i)>0,\;\;\forall h\in\Wf.
\end{equation}
\end{prop} 
\begin{proof} Let $\ring{\Pf}_n^-=\{h\in\Pf_n\mid h<0\}$. Then $\Wf$ and $\ring{\Pf}_n^-$ are disjoint
relatively open convex cones  of $\Pf_n$ with vertex in $0$ and therefore 
(see e.g. \cite[Thorem 2.7]{Solt})
are separated by 
a hyperplane, defined by a linear functional $\psiup$, which is positive on $\Wf$
and negative on $\ring{\Pf}_n^-$. 
Being negative on $\ring{\Pf}_n^-$, by Lemma~\ref{lem6.5}, 
$\psiup$  has the form~\eqref{q5.1}. 
This implies that all elements of $\overline{\Wf}\cap\bar{\Pf}^-_n$ are degenerate.
Since $\overline{\Wf}\cap\bar{\Pf}^-_n$ is a cone, all its elements of 
maximal rank belong to its relative interior and
have the same kernel, say $U\subset\C^n$, whose positive dimension we denote by $r$. 
In fact, for a pair of negative semidefinite forms $h_1,h_2$, we have $\ker{(h_1+h_2)}=\ker{h}_1\cap\ker{h}_2$.
The statement follows by applying Proposition~\ref{prc1} to $\overline{\Wf}|_U=\{h|_U\mid h\in\overline{\Wf}\}$,
which is a closed cone in $\Pf_r$ in which all nonzero elements have a nonzero positive index of inertia.
In fact, if there is a nonzero $h\in\overline{\Wf}$ whose restriction to $U$ is seminegative, and $h_0$ 
is an element of maximal rank in the cone $\overline{\Wf}\cap\bar{\Pf}^-_n$, then, for $C>0$ and large,
$h+Ch_0$ would be a negative definite element in $\overline{\Wf}\cap\bar{\Pf}^-_n$. 
\end{proof}
\begin{prop}
 Let $\Wf$ be a cone in $\Pf_n,$ with the property
 that all its elements of maximal rank  have a nonzero positive index of inertia. 
 Then all forms in $\overline{\Wf}\cap\bar{\Pf}_n^-$ are degenerate; those of maximal rank have all
 the same kernel, of dimension $r>0$, which contains a basis $e_1,\hdots,e_r$ such that
\begin{equation}\label{4.3}
 {\sum}_{i=1}^r{h}(e_i,e_i)\geq{0}, \;\; \forall h\in\Wf.
\end{equation}
\end{prop} 
\begin{proof}
 Let $\ring{\Pf}_n^+=\{h\in\Pf_n\mid h>0\}$. Then $\Wf+\ring{\Pf}_n^{+}$ is an open cone in $\Pf_n$
 such that all its elements have a nonzero positive index of inertia.
 \par
 Since $\overline{\Wf+\ring{\Pf}_n^{+}}\cap\bar{\Pf}_n^-=(\overline{\Wf}+\bar{\Pf}_n^+)\cap\bar{\Pf}_n^-
 =\overline{\Wf}\cap\bar{\Pf}_n^-$, we know from  
 Proposition~\ref{pp6.1} that all elements of maximal rank in $\overline{\Wf}\cap\bar{\Pf}_n^-$
 have the same kernel $U$, which is a subspace of $\C^n$ of positive dimension $r$ and contains
 a basis $e_1,\hdots,e_r$ for which
\begin{equation*}
 {\sum}_{i=1}^r h(v_i,v_i)>0,\quad\forall h\in \Wf+\ring{\Pf}_n^{+}.
\end{equation*}
This implies \eqref{4.3}.
\end{proof}
Analogous results can be given to characterize cones of Hermitian forms having some given amount 
of positive (or negative) eigenvalues. In this case we need to consider the behavior of the restriction
of forms to subspaces of $\C^n$. We use the notation $\Gr_{\!{h}}(\C^n)$ for the Grassmannian of complex
linear $h$-planes of $\C^n$. 
\begin{prop}
 Let $\Wf$ be a proper closed convex 
 cone in $\Pf_n$, with vertex in $0$ and $q$ an integer with $0<q\leq{n}$.
 Assume that every nonzero form in $\Wf$ has a positive index of inertia $\geq{q}$. Then,
for every $V\in\Gr_{\! n-q+1}(\C^n)$, we can find a basis
 $v_1,\hdots,v_{n-q+1}$ of $V$ such that 
\begin{equation}
 {\sum}_{i=1}^{n-q+1}h(v_i,v_i)\geq{0}.
\end{equation}
\end{prop} 
\begin{proof}
It suffices to apply Proposition~\ref{prc1} to the restrictions to $V\in\Gr_{\! n-q+1}(\C^n)$ of  
the forms in $\Wf$. By the assumption, $h|_V$ has a nonzero positive index of inertia for
all $h\in\Wf\setminus\{0\}$.
\end{proof}
An analogous statement to Proposition~\ref{pp6.1} can be formulated for relatively open convex
cones of Hermitian forms with positive index of inertia $\geq{q}$.
\begin{prop}\label{pr4.11}
 Let $\Wf$ be a relatively open convex cone in $\Pf_n$ and assume that each $h$ in $\Wf$ 
 has a positive index of inertia $\geq{q}$, for an integer  $0<q\leq{n}$. Then 
 for every $V\in\Gr_{\! n-q+1}(\C^n)$
 we can find an integer 
 $r_V>0$ and linearly independent $v_1,\hdots,v_{r_V}\in{V}$ such that 
 \begin{equation}\label{4.4}
 {\sum}_{i=1}^{r_V}h(v_i,v_i)>0,\quad\forall h\in\Wf.
\end{equation}
\end{prop} 
\begin{proof} For every  $V\in\Gr_{\! n-q+1}(\C^n)$, the set $\Wf_V=\{h|_V\mid h\in\Wf\}$ is a
relatively open convex cone of $\Pf_{n-q+1}$ such that all of its elements $h|_V$ have a
nonzero positive index of inertia. The thesis follows by applying 
Proposition~\ref{pp6.1} to $\Wf|_V$.
\end{proof}
\begin{prop}\label{pr4.12}
 Let $W$ be a convex cone in $\Pf_n$ such that the elements of maximal rank of $W$
 have a positive index of inertia $\geq{q}$ ($q$ is an integer with $0<{q}\leq{n}$). 
 Then for every $V\in\Gr_{\! n-q+1}(\C^n)$ we can find an integer 
 $r_V>0$ and linearly independent $v_1,\hdots,v_{r_V}\in{V}$ such that 
 \begin{equation}\label{4.6}
 {\sum}_{i=1}^{r_V}h(v_i,v_i)\geq {0},\quad\forall h\in\Wf.
\end{equation}
\end{prop} 
\begin{proof}
 It suffices to apply Proposition~\ref{pr4.11} to $\Wf+\ring{\Pf}_n^{+}$ and note that \eqref{4.4}
 for all $h\in\Wf+\ring{\Pf}_n^{+}$ implies \eqref{4.6} for all $h\in\Wf$.
\end{proof}
\begin{rmk}
 The positive integer $r_V$ of Propositions~\ref{pr4.11},\ref{pr4.12} 
 is the dimension of the kernel of any form of maximal rank in 
 $\overline{\Wf}_V\cap\bar{\Pf}_{n-q+1}^-$.
\end{rmk}
\section{Notions of pseudo-concavity} \label{s7}
In \cite{HN06} it was proved that the Poincar\'e lemma
for the tangential Cauchy-Riemann complex
of locally $CR$-embeddable $CR$ manifolds
fails in the degrees corresponding to the indices of inertia of its scalar Levi forms of maximal rank. 
On the other hand,  in \cite{33}   
it was shown that the Lefschetz hyperplane section theorem for $q$-dimensional complex submanifolds
generalizes to 
weakly-$q$-pseudo-concave  $CR$ submanifolds
of complex projective spaces.\par
This suggests to seek for suitable weakening of the pseudoconcavity conditions 
to allow  degeneracies of the Levi form. 
A natural condition 
of  weak $1$-pseudo-concavity is to require that 
no semi-definite
scalar Levi form has maximal rank.  Under some genericity assumption, 
by using Proposition~\ref{pr4.12}, this translates into the fact that 
$\Kf$ 
is non-trivial. Indeed, this hypothesis 
implies 
maximum modulus and unique continuation results 
analogous to those for holomorphic functions of 
one complex variable.
We 
expect that properties that are peculiar to holomorphic functions of 
 \textit{several} complex variables 
 would generalize to $CR$ functions 
under 
 suitable \textit{(weak) $2$-pseudo-concavity conditions.} 
 This motivates us to give below a tentative list of 
 conditions, motivated  partly by the discussion in
\S\ref{cv} and partly by the results of the next sections.
\par \smallskip
\begin{ntz}\label{ntz7.1}
If $\Vf\subset\Zf$ is a distribution of complex vector fields on $\Omega^{\mathrm{open}}\subset{M}$,
we use the notation $\Kf_{\Vf}$ for the semi-positive tensors ${\sum}_{i=1}^rZ_i\otimes\bar{Z}_i$
of $\Kf$
with $Z_i\in\Vf$.
\end{ntz}
\begin{defn} Let $p_0\in{M}$. We say that $M$ is 
\begin{enumerate}
\item[$(\Psi_{p_0}^s(q))$:] \emph{strongly-$q$-pseudo-concave} at $p_0$ if all $\Lf_{\xiup}$, with 
$\xiup\in{H}_{p_0}^0M\setminus\{0\}$,  
are nonzero and have 
Witt index $\geq{q}$; 
 \item[$(\Psi_{p_0}^w(q))$:] \emph{weakly-$q$-pseudo-concave} at $p_0$ 
 if its scalar Levi forms of maximum rank
 at $p_0$  have  
 Witt 
 index 
 $\geq{q}$;
 \item[$(\Psi_{p_0}^e(q))$:] \emph{essentially-$q$-pseudo-concave} at $p_0\in{M}$ if, for every distribution
 of smooth complex vector fields $\Vf\subset\Zf$, 
 of rank $n\! -\! q\! +\! 1$, defined on an open neighborhood $U$ of $p_0$, 
 we can find an open neighborhood $U'$ of $p_0$ in $U$ and a $\tauup\in\Kf_{\Vf}^{n-q+1}(U')$.
 \item[$(\Psi_{p_0}^{e^*}(q))$:] \emph{essentially$^*$-$q$-pseudo-concave} 
 at $p_0\in{M}$ if, for every distribution
 of smooth complex vector fields $\Vf\subset\Zf$,
 of rank $n\! -\! q\! +\! 1$, defined on an open neighborhood $U$ of $p_0$, 
 we can find an open neighborhood $U'$ of $p_0$ in $U$ and a 
 $\tauup\in\Kf_\Vf(U')$.
\end{enumerate}
We drop the reference to the point $p_0$ when the property is valid at all points of $M$. \par
We also consider the (global) condition
\begin{enumerate}
 \item[$(\Psi^{we}(q))$] For all $p\in{M}$ and $\Vf\subset\Zf$ of rank $n-q+1$ on a neighborhood  
 $U$ of $p$,  ${\bigcup}_{p'\in{U}}\Kf_{\Vf,p'}$ is a bundle with nonempty fibers
and such that for every sequence $\{p_\nu\}\subset{M}$, converging to  $p\in{M}$,
every  
$\tauup\in\Kf_{\Vf,\,p}$ is a cluster point of $\cup_{\nu}\Kf_{\Vf,\,p_\nu}$.
\end{enumerate}
\end{defn}
Recall that, according to the notation introduced on page \pageref{kappa},
the elements of $\Kf(U')$ are different from zero at each point of $U'$.
\begin{rmk}
 If $q>1$, then $\Psi_{p_0}^{\star}(q)\Rightarrow \Psi_{p_0}^{\star}(q-1)$ for
 $\star=s,w,e,e^*$,
and (cf.~Proposition~\ref{prc1} and \cite[\S{2}]{HN00})
\begin{equation*}
 \Psi^w(q)\Leftarrow \Psi^s(q)\Rightarrow \Psi^e(q)\Rightarrow\Psi^{e^*}(q),\quad
 \text{for $q\geq{1}.$}
\end{equation*}
\end{rmk}
 
 \begin{lem}
 Assume that $M$ 
 is essentially-$q$-pseudo-concave. Then, for every rank  $n\! -\! q\! +\! 1$
distribution  $\Vf\subset\Zf$ on an $\Omega^{\mathrm{open}}\subset{M}$,
we can find a global section $\tauup\in\Kf^{(n-q+1)}_\Vf(\Omega)$.
\end{lem} 
\begin{proof}
By the assumption, for each $p\in\Omega$, there is an $U^{\mathrm{open}}\subset\Omega$
with $p\in{U}_p$ and $\tauup_p={\sum}_{i=1}^{n-q+1}Z_i\otimes\bar{Z}_i\in\Kf^{(n-q+1)}(U_p)$ with 
$Z_i\in\Vf(U_p)$. The global $\tauup$ can be obtained by gluing together the $\tauup_p$'s
by a nonnegative smooth partition of unity on $\Omega$ subordinate to the covering
$\{U_p\}$.
\end{proof}
In the same way we can prove 
 \begin{lem}
 Assume that $M$ 
 is essentially$^*$-$q$-pseudo-concave. Then, for every rank  $n\! -\! q\! +\! 1$
distribution  $\Vf\subset\Zf$ on an $\Omega^{\mathrm{open}}\subset{M}$,
we can find a global section $\tauup\in\Kf_\Vf(\Omega)$. \qed
\end{lem} 
\begin{exam} Let $F_{h_1,\hdots,h_r}(\C^m)\subset\Gr_{\! h_1}(\C^m)\times
\cdots\times\Gr_{\! h_r}(\C^m)$ denote the complex flag manifold consisting
of the $r$-tuples  $(\ell_{h_1},\hdots,\ell_{h_r})$ with $\ell_{h_1}\subsetneqq \cdots
\subsetneqq \ell_{h_r}$, for an increasing sequence $1\leq{h}_1<\cdots<h_r<m$. 
Here, as usual, $\ell_h$ is a generic $\C$-linear subspace of dimension $h$ of $\C^m$.
\par
For an 
increasing sequence of integers $1\leq{i}_1<i_{2}<\cdots{i}_\nu<m$, of length $\nu\geq{2}$, 
we define the $CR$-sub-manifold
$M$ of $F_{i_1,i_3,\hdots}(\C^m)\times{F}_{i_2,i_4,\hdots}(\C^m)$ 
consisting of pairs $(({\ell}_{i_1},\ell_{i_3},\cdots),(\ell_{i_2},\ell_{i_4},\cdots))$ with 
$\overline{\ell}_{i_h}\subset\ell_{i_{h+1}}$ for $0<h<\nu$. 
Set
\begin{equation*}
 d_0=i_1,\; d_1=i_2-i_1,\; \cdots,\; d_h=i_{h+1}-i_h,\;\cdots d_{\nu-1}=i_\nu-i_{\nu-1},\; d_{\nu}=m-i_{\nu}.
\end{equation*}
This $M$ is a minimal 
(i.e. 
$\Zf(M)+\overline{\Zf}(M)$ and their iterated commutators yield
all complex vector fields on $M$),
compact 
$CR$ manifold of $CR$-dimension $n$  and $CR$-codimension $k$, with 
\begin{equation*}
n={\sum}_{i=0}^{\nu-1} d_id_{i+1},\quad k=2{\sum}_{ 
\begin{smallmatrix}
 1\leq{i}<j\leq{\nu}\\
 j-i\geq{2}
\end{smallmatrix}
 }d_id_j,
\end{equation*}
as was explained in 
 \cite[\S{3.1}]{MN98}. Then, with $q={\min}_{1<i<\nu}d_i$, our $M$ is essentially,
 but not strongly, $q$-pseudo-concave when $\nu\geq{3}$, because the non-vanishing 
 scalar Levi forms generate at each point a subspace of dimension
 $2{\sum}_{i=1}^{\nu-2}d_id_{i+2}<k$.
 \par
 In \cite{MN98} several classes of homogeneous compact $CR$ manifolds are discussed, from which
 more examples of essentially, but not strongly, $q$-pseudo-concave manifolds can be extracted. 
\end{exam}
\begin{exam} 
Let us consider the $11$-dimensional real vector space 
$\Wf$ consisting of $4\times{4}$ Hermitian symmetric matrices of the form
\begin{equation*}
h=\begin{pmatrix}
 A & B\\
 B^*&-A
\end{pmatrix}
 \quad 
\begin{gathered}
 \text{with $A,B\in\C^{2\times{2}}$, $A=A^*,$ 
 $\mathrm{trace}(A)=0$.}
\end{gathered}
\end{equation*}
We claim that all nonsingular
 elements of $\Wf$ have Witt index two.
In fact, for an element $h$ of $\Wf$, either $A=0$, or $A$ is nondegenerate. If $A=0$, the matrix $A$
is nondegenerate iff $\det(B)\neq{0}$, and in this case the Witt index is two as the two-plane of
the first two vectors of the canonical basis of $\C^4$ is  totally isotropic. If $A\neq{0}$, a permutation
of the vectors of the canonical basis of $\C^4$ transforms $h$ into a Hermitian symmetric matrix $h'$
with 
\begin{equation*}
 h'= 
\begin{pmatrix}
 C & D \\
 D^* & -C
\end{pmatrix},
\end{equation*}
for a positive definite Hermitian symmetric
$C\in\C^{2\times{2}}$. By a linear change of coordinates in  $\C^2$, the positive definite
$C$ reduces to the $2\times{2}$ identity matrix $I_2$. This yields a change
of coordinates in $\C^4$ by which $h'$ transforms into 
\begin{equation*}
 h''= 
\begin{pmatrix}
 I_2 & E\\
 E^* & -I_2
\end{pmatrix},\quad \text{with $E\in\C^{2\times{2}}$.}
\end{equation*}
For a matrix of this form, we have,
for  $v,w\in\C^2$, 
\begin{align*}
 h'' 
\begin{pmatrix}
 v\\ w
\end{pmatrix}=0 \Leftrightarrow 
\begin{cases}
 v+Ew=0,\\
 E^*v-w=0
\end{cases} \Leftrightarrow 
\begin{cases}
 v+EE^*v=0,\\
 w=E^*v
\end{cases} \Leftrightarrow 
\begin{cases}
 v=0,\\
 w=0.
\end{cases}
\end{align*}
Therefore, all $h''$ of this form are nonsingular and their Witt is independent of $E$
and  equal to two. 
This shows that all $h\in\Wf$ with $A\neq{0}$ are nonsingular with Witt index two. Thus
the set of singular matrices of $\Wf$ is
\begin{equation*}
 \left.\left\{ 
\begin{pmatrix}
 0 & B\\
 B^*&0
\end{pmatrix}\right| \det(B)=0\right\},
\end{equation*}
which is the cone of the nonsingular quadric of the $3$-dimensional projective space.\par 
If we take a basis
$h_1,\hdots,h_{11}$ of $\Wf$, the quadric $M$ of $\C^{14}=\C^4_z\times\C^{11}_w$, defined by 
the equations 
\begin{equation*}
 \re(w_i)=h_i(z,z),\; 1\leq{i}\leq{11},
\end{equation*}
is a $CR$ manifold of type $(4,11)$ which is weakly and weakly$^*$-$2$-pseudo-concave,
but not strongly or essentially-$2$-pseudo-concave.\par
We obtain examples of   $CR$ manifolds 
$M=\{(z,w)\in\C^4\times\C^7\mid \re(w_i)=h_i(z,z),\; 1\leq{i}\leq{7}\}$, of type $(4,7)$ 
and \textit{strongly} $2$-pseudoconcave
by
requiring that $h_1,\hdots,h_7$ be a basis 
either of the subspace $\Wf'$ of $\Wf$ in which $B$ is traceless and symmetric, or of the
$\Wf''$ in which $B$ is  quaternionic.
\end{exam}
\begin{exam}
 Let $M$ be the minimal orbit of $\SU(p,p)$ in the complex flag manifold $F_{1,2p-2}(\C^{2p})$,
 for $p\geq{3}$. 
 Its points are the pairs $(\ell_1,\ell_{2p-2})$ consisting of an isotropic line $\ell_1$
 and a $(2p\!-\!{2})$-plane $\ell_{2p-2}$ with $\ell_1\subset\ell_{2p-2}\subset\ell_1^\perp$, where
 perpendicularity is taken with respect to a 
 fixed Hermitian symmetric form of Witt index $p$ on $\C^{2p}$. \par
 Then
 $M$ is a compact $CR$ submanifold
 of $F_{1,2p-2}(\C^{2p})$, 
 of $CR$ dimension $(2p\! -\! 3)$ and $CR$ codimension $(4p\! -\! 4)$, which is essentially $1$-pseudo-concave
 and, when ${p>3}$, 
 weakly and weakly$^*$-$(p\!-\! 2)$-pseudo-concave, but not essentially-$2$-pseudo\-concave.
\end{exam}
\subsection{Convexity/concavity at the boundary and weak pseudoconcavity}
Let us comment on 
the notion of $1$-convexity/concavity at a boundary point of a domain
$\Omega$ of \S\ref{sechopf}  in the light
of the discussion on Hermitian forms of \S\ref{cv}. \par
Let $\rhoup$ be a real valued smooth function on $\Omega^{\mathrm{open}}\subset{M}$ and 
$p_0$ a point of $\Omega$ with the property 
that, for each $id\xiup_{p_0}$ in ${H}^{1,1}_{p_0}(\rhoup)$,  
the restriction of $id\xiup_{p_0}$ to the space
\mbox{$\{Z_{p_0}\in{T}^{0,1}_{p_0}M\mid
Z_{p_0}\rhoup=0\}$} 
has a nonzero positive index of inertia. The positive multiples of these
Hermitian symmetric forms make a relatively open convex cone 
$\Wf$ in the space $\Pf_{n-1}$ of Hermitian symmetric forms on $T^{0,1}_{p_0}M\cap\ker{d}\rhoup(p_0).$
By Proposition~\ref{pp6.1},
 we can find an $r>0$ and $\tauup_0\in{H}^{1,1,(r)}_{p_0}M$ such that 
\begin{equation*}
 id\xiup(\tauup_0)>0,\quad \forall \xiup\in\Af_1(\Omega),\;\;\text{s.t.}\;\; d\rhoup(p_0)+i\xiup_{p_0}
 \in{{T}^*}^{1,0}_{\!{p_0}}M.
\end{equation*}
Since $H^{1,1}_{p_0}(\rhoup)$ is affine with underlying vector space $\{\Lf_\eta\mid\eta\in{H}^0_{p_0}M\}$,
it follows that actually $\tauup_0\in\Kf^{(r)}_{p_0}$. 
The same argument applies to the case of a nonzero negative index of inertia. \par 
Thus, by Lemma~\ref{lem5.18}, the condition
for
$\Omega_{\rhoup(p_0)}=\{p\in\Omega\mid \rhoup(p)<\rhoup(p_0)\}$ to be 
strongly-$(1)$-convex, or strongly-$(1)$-concave 
at $p_0$ is that
\begin{equation} \exists \tauup_0\in\Kf_{\ker{d\rhoup},\,p_0}\;\;\text{such that}\;\;
\begin{cases}
 dd^c\rhoup(\tauup_0)>0, &\text{(strongly-$1$-convex)},\\
  dd^c\rhoup(\tauup_0)<0, &\text{(strongly-$1$-concave).}
   \end{cases}
\end{equation}
A glitch of the notion of strong-$1$-convexity (resp.\,-concavity)
is that it is not, in general, stable under small perturbations. 
This can be ridden out by adding the global 
assumption of essential-$2$-pseudo-concavity of $M$. 
Set, for simplicity of
notation, $\rhoup(p_0)=0$ and $d\rhoup(p_0)\neq{0}$. 
\begin{prop}\label{prop7.7}
Suppose that $M$ is   essentially-$2$-pseudo-concave and that  $\Omega_0=\{p\in\Omega\mid
\rhoup(p)<0\}$ is strongly-$1$-concave at $p_0\in\partial\Omega_0$. Then 
\begin{enumerate}
 \item We can find $\tauup_0\in\Kf_{\ker{d\rhoup},\,p_0}^{(n-1)}$ 
 such that $dd^c\rhoup(\tauup_0)<0$;
 \item We can find an open neighborhood $U$ of $p_0$ in $\Omega$ such that at every
 $p'\in{U}$ the open set $\Omega_{\rhoup(p')}=\{p\in\Omega\mid  \rhoup(p)<\rhoup(p')\}$
 is smooth and strongly-$1$-concave at $p'$. \qed
\end{enumerate}
\end{prop}
\section{Cauchy problem for \textsl{CR} functions - Uniqueness} \label{s8}
In this section we 
discuss uniqueness for 
the initial value problem for $CR$ functions, with
data on a non-characteristic
smooth initial hypersurface $N\subset{M}$. \par

\smallskip
Uniqueness is well understood when $M$ is a $CR$ submanifold of a complex manifold (see
e.g. \cite{Sm}). Let $\Omega\subset{M}$ be an open neighborhood 
of a non-characteristic point $p_0$ of
$N$, such that $\Omega\setminus{N}$ 
is the union of two disjoint connected 
components~$\Omega^\pm$.
\begin{prop} 
Assume that $M$ is a minimal $CR$ submanifold of a complex manifold ${\Xf}$.
If $f\in\Ot_{\!{M}}(\Omega^+)\cap\Co^0(\bar{\Omega}^+)$ 
and $f|_N$ vanishes on an open neighborhood
of a non-characteristic point
$p_0$ of $N$, then $f\equiv{0}$ on $\Omega^+$. \qed
\end{prop}

We have a similar statement for $CR$ distributions.
\begin{prop} Assume that $M$ is either
a real-analytic $CR$ manifold, 
or a 
$CR$ submanifold of a complex manifold ${\Xf}$
that is minimal at every point. 
Let
$N$ be a $\Zf$-non-characteristic hypersurface of $M$, 
such that $M\setminus{N}$ is the
union of two disjoint connected open subsets $M_\pm$. 
Then there
is an open neighborhood $U$ of $N$ in $M$ such that 
any $CR$
distribution on $M_+$ having vanishing boundary values on $N$, 
vanishes on $U\cap{M}_+$. 
\end{prop} 
\begin{proof} An $f\in\cD'(M_+)$ is $CR$ if $Zf=0$ in $M_+,$ 
in the sense of distributions,
for all $Z\in\Zf(M)$. We say that $f$ has
\textit{zero boundary value} on $N$ 
if for each $p\in{N}$ we can find an open neighborhood $U_p$
of $p$ in $M$ and a $CR$-distribution $\tilde{f}\in\cD'({U}_p)$ which extends 
$f|_{M_+\cap{U}_p}$ 
and is zero on $U_p\setminus\bar{M}_+$. 
Note that, since $N$ is non-characteristic, 
all $CR$ distributions defined on a neighborhood of $N$, admit a \textit{restriction} to $N$.\par
The case where $M$ is a real-analytic $CR$ manifold reduces to the classical Holmgren 
uniqueness theorem. 
\par
In the other case, where
$M$ is 
$\Ci$ smooth, but is assumed to be
minimal, 
we first choose a slight deformation $N_d$ of $N$
such that $N_d$ is contained in $\overline{M^+}$
and coincides with $N$ near $p$.
Moreover we can achieve that 
the CR orbit $\cO(p,N_d)$ of $p$
in $N_d$ intersects $N_d\cap M^+$.
Since $M^+$ is minimal at every point,
CR distributions holomorphically extend 
to open wedges attached to $M^+$.
In particular this holds for the boundary
value of $f|_{M_d^+}$
($M_d^+$ being the side of $N_d$
containing $M^+$) at any point
of $N_d\cap M^+$.

Using that wedge extension propagates 
along CR orbits we get wedge extension 
from $N_d$ at $p$. 
Examining how the wedges are constructed
by analytic disc techniques, one more precisely 
obtains a neighborhood V of $p$ in $\overline{M^+}$ 
and an open truncated cone
$C\subset\C^n$ such that
$\tilde{f}$ holomorphically extends to  
$W_N=\bigcup_{z\in V\cap N} (z+C)$, 
and $f$ to $W^+=\bigcup_{z\in V\cap M^+} (z+C)$.
The idea is to work with analytic discs attached
to (deformations of) $N_d$ and 
to nearby hypersurfaces of $M$.

Since $\tilde{f}$ is the boundary value of $f$,
the two extensions glue to a single function 
$F\in\cO({W_N\cup W^+})$. 
On the other hand, $F$ is zero on $W_N$
(since $\tilde{f}$ vanishes near $p$) and
thus on $W^+$,
by the unique continuation of holomorphic functions. 
Finally $f$, being the boundary value of $F$,
has to vanish on $N\cap M^+$.
%
\end{proof}
\begin{rmk}\rm
Thanks to the extension result proved in \cite{Jo1,Mk1}, see also \cite{MePo},
it suffices to assume that $M^+$ is {\it globally minimal},
i.e.~that $M^+$ consists of only one CR orbit.
\end{rmk}

For an embedded $CR$ manifold with property $(H)$,
uniqueness results
can be derived from 
Proposition~\ref{propo2.3}.
Indeed, in this case, a $CR$ function defined on a neighborhood in $M$ 
of a point $p_0\in{N}$ and whose restriction to $N$ has a zero of infinite order at $p_0$,
also has a zero of infinite order at $p_0$ as a function on $M$ and then is zero on the connected
component of $p_0$ in its domain of definition by the strong unique continuation principle.
\par\smallskip
The situation is quite different for 
abstract $CR$ manifolds: there are examples 
of pseudo-convex
$M$ 
on which there are nonzero 
smooth $CR$ functions vanishing on an open subset (see e.g. \cite{Ros}).
Here, for the pseudo-concave case, we give a uniqueness result 
which is similar to those of 
\cite{DCN, HN00, HN03}, 
but more general, because 
we do not require the existence 
of sections $\tauup$ of 
$\Kf^{(n)}$,
i.e. we drop the 
rank requirement, but we  assume that 
the initial hypersurface
$N$ is non-characteristic with respect to
the sub-distribution $\Theta$ of $\Zf$, which was defined in  \S\ref{s3}.
\par
In this context we can slightly generalize $CR$ functions by considering,
for a given $\tauup\in\Kf(M)$,
functions $f$ on $M$ satisfying 
\begin{equation} \label{eq:3.1}
\begin{cases}
 f\in{L}^2_{\mathrm{loc}}(M),\;\; \forall Z\in\tilde{\Theta}, 
 Zf\in{L}^2_{\mathrm{loc}}(M)\;\;
\text{and} \;\exists \kappaup_Z\in{L}^{\infty}_{\mathrm{loc}}(M,\R)\;\; \\ \; \text{such that}
 \quad \;\; |(Zf)(p)|\leq
 \kappaup_Z(p)|f(p)| \;\; \text{a.e. on $M$}.
\end{cases}
 \end{equation}
 \par 
 Condition \eqref{eq:3.1}, with $\Zf(M)$ instead of $\tilde{\Theta}(\tauup)$,  
 naturally arises when we consider 
 $CR$ sections of a complex $CR$ line
 bundle (see \cite[\S{7}]{HN00}). \par
 We note that the hypersurface $N$ is non-characteristic at a point $p_0$ with respect to the
 distribution $\Theta$ if it is non-characteristic at $p_0$ for $\Theta(\tauup)$ for some
 $\tauup\in\Kf(M)$. 
\begin{prop}\label{prp6.1} 
Let $\Omega^{\mathrm{open}}\subset{M}$ 
and $N\subset\partial\Omega$
a smooth $\Theta$-non-characteristic hypersurface in $M$.
Then 
there is a neighborhood $U$ of $N$ in $M$ such that 
any solution $f$ of 
 \eqref{eq:3.1}, which is continuous on $\bar{\Omega}$
 and vanishes on $N$, is zero on $U\cap\Omega$.
\end{prop} 
\begin{proof}
We note that the assumption of constancy of rank in unessential and
never used in the proof of \cite[Theorem 4.1]{HN03}.
We reduce to that situation by considering the 
$\tilde{\Theta}$-strucute on $M,$
defined by the distribution of \eqref{theta3.6}, 
after we make the following observation. Since the statement is local,
we can assume that $N$ 
splits  $M$ into two closed half-manifolds $M_\pm$, with $\Omega=M_-$
and $\partial\Omega=N$. A continuous solution
$f$ of \eqref{eq:3.1} in $M_-$ vanishing on $N$, when
extended by $0$ on $M_+$, defines 
a continuous solution $\tilde{f}$ of \eqref{eq:3.1} in $M$ 
whith $\supp\tilde{f}\subset\bar{M}_-$.
In fact, since $L\in\tilde{\Theta}(M)$ is first order, $L\tilde{f}$ 
equals $Lf$ 
on $M_-$ and  $0$
on $M_+$, as one can easily check by integrating by parts and using 
the identity of weak and strong extensions of \cite{F44}. Hence 
$\tilde{f}$ still satisfies \eqref{eq:3.1}
and vanishes on  an open subset of $M$. By proving Carleman estimates, similar to those  
in \cite[Theorem~5.2]{HN00},
we obtain that $\tilde{f}$ vanishes along the Sussmann leaves of $\tilde{\Theta} 
$ 
transversal
to $N$ 
(see \cite{HN03, Su73}).
These leaves 
fill a neighborhood of $N$ in $M$, 
where
$\tilde{f}$ vanishes. This proves our contention.
 \end{proof}

\begin{rmk}
 Note that $\C^n\times\R^r$ is weakly pseudo-concave (but not essentially pseudo-concave).
 Thus we need the genericity assumption \eqref{eq:8.1} to get uniqueness in this case. 
 The uniqueness for the non-characteristic
 Cauchy problem in the case of 
 a single partial differential operator of \cite{CH79, ST74} 
 may be considered 
 a 
 special case of this proposition, when the $CR$ dimension is one. 
 \end{rmk}
Uniqueness in the case where $N$ can be 
 characteristic for $\Theta 
 $, but not for 
 $\Zf 
 $,
 will be obtained by adding a pseudo-convexity hypothesis. 

First we prove a Carleman-type estimate.

\begin{lem}
 Let $\tauup$ 
 be a section of $\Kf$ and 
 $\psiup$ 
 a real valued smooth function on~$M$. Then there is a smooth real valued
 function $\kappa$ on $M$ such that 
 \begin{align}\label{eq8.3}
 \|\exp(t\psiup)L_0f\|_0^2\! +\! {\sum}_{i=1}^r\|\exp(t\psiup)Z_if\|_0^2
\geq\int(2t\cdot 
dd^c\psiup(\tauup)+\kappa)
\, |f|^2{e}^{2t\psiup}d\mu ,\;\\
 \notag \forall f\in\Cic(M),\; \forall t>0.
\end{align}
\end{lem}
Here the $L^2$-norms and the integral are defined
 by utilizing  the smooth measure $d\mu$
associated to  a fixed Riemannian metric on $M$.
\begin{proof}
Let $\tauup={\sum}_{i=1}^r{Z}_i\otimes\bar{Z}_i$, 
${\sum}_{i=1}^r[Z_i,\bar{Z}_i]=\bar{L}_0-L_0$,
with $Z_i,L_0\in\Zf(M)$. 
We will indicate by $\kappa_{\!{1}},\kappa$ smooth functions on $M$ which only depend on
$Z_1,\hdots,Z_r$. 
For $f\in\Cic(M)$, and a fixed $t>0$, set $v=f\!\cdot\!\exp(t\psiup)$. 
Integration by parts yields
\begin{align*}
{\sum}_{i=1}^r \|Z_iv-{t}{v}{Z}_i\psiup\|_0^2&= {\sum}_{i=1}^r\|Z_i^*v-{t}{v}\bar{Z}_i\psiup\|_0^2
+\int{\sum}_{i=1}^r[{Z}_i,\bar{Z}_i]v\cdot\bar{v}\,d\mu \\
& \qquad\qquad\qquad\qquad
+  \re\int\left(\kappa_0+{\sum}_{i=1}^r
2{t}(Z_i\bar{Z}_i\psiup)\right)|v|^2d\mu ,
\end{align*} 
where 
the superscript star stands for formal adjoint 
with respect to the Hermitian scalar product of $L^2(d\mu)$.
For the second summand in the right hand side
we have 
\begin{align*}
\int{\sum}_{i=1}^r[{Z}_i,\bar{Z}_i]v\cdot\bar{v}\,d\mu &=\int \bar{L}_0v\cdot\bar{v}\,d\mu
- \int{L}_0v\cdot\bar{v}\,d\mu  \\
 &=-\int{L}_0v\cdot\bar{v}\,d\mu-\int{v}\cdot\overline{L_0v}\,d\mu-\int\kappa_1|v|^2d\mu\\
 &= -2\re\int{L}_0v\cdot\bar{v}\,d\mu-\int\kappa_1|v|^2d\mu\\
& \geq - {2} \|L_0v-{t}{v}L_0\psiup\|_0\|v\|_0-\int(\kappa_1+{2{t}\re{L}_0\psiup})|v|^2d\mu\\
&\geq -\|L_0v-{t}{v}L_0\psiup\|_0^2-\int(1+\kappa_1+{2{t}\re{L}_0\psiup})|v|^2d\mu.
\end{align*}
Therefore we obtain the estimate 
\begin{align*}
& \|L_0v-{t}{v}L_0\psiup\|_0^2+{\sum}_{i=1}^r\|Z_iv-{t}{v}Z_i\psiup\|_0^2\\ &\qquad\quad \geq 
 \int \big({t}[Z_i\bar{Z}_i+\bar{Z}_i{Z}_i]\psiup-2{t}(\re{L}_0)\psiup-\kappa_2 \big)\,|v|^2\,d\mu
 =\int (2{t}{P}_{\tauup}\psiup+\kappa)|v|^2d\mu.
\end{align*}
By Proposition~\ref{p4.6},
this yields \eqref{eq8.3}.
\end{proof}
From the Carleman estimate \eqref{eq8.3} we obtain a uniqueness result 
\textit{under convexity conditions},
akin to the one of
\cite[\S{28.3}]{Horm85} for a scalar p.d.o.
\begin{prop} \label{p8.7}
Assume there is a section $\tauup\in\Kf$
 and  $\psiup\in\Ci(M,\R)$ such that 
\begin{equation} d\psiup(p_0)\neq{0},\quad 
dd^c\psiup(\tauup)>0.
\end{equation}
Then there is an open neighborhood $U$ of $p_0$ in $M$ with the property that
any solution $f$  of \eqref{eq:3.1}
 which vanishes a.e. on $U\cap\{p\mid\psiup(p)>\psiup({p_0})\}$ also
 vanishes a.e. on~$U$. \qed
\end{prop}
\begin{rmk}
 In fact, it suffices to require that \eqref{eq:3.1} is satisfied by the operators $Z_1,\hdots,Z_r,L_0$.
\end{rmk}
Let $\Omega$ be an open domain in $M$, and $p_0\in\partial\Omega$ a smooth point of
the boundary. 

\begin{prop}\label{prop8.8}
 If $\Omega$ is either $\Theta$-non-characteristic or strictly $1$-convex at $p_0$
 (according to Definition~\ref{d4.4}), 
 then any $f$ satisfying
 \eqref{eq:3.1} in $\Omega$,
 and having zero boundary values on a neighborhood of $p_0$ in $\partial\Omega$,
 is $0$ a.e. on the intersection of $\Omega$ with a neighborhood of $p_0$ in $M$.\qed
\end{prop}
\begin{proof}
With $P_{\!\tauup}$
defined by \eqref{1.4b}, \eqref{2.2}, and a real parameter $s$, we have 
\begin{align*}
e^{-s\psiup} P_{\!\tauup}({e}^{s\psiup})=s\left(\tfrac{1}{2}{\sum}_{i=1}^r(Z_i\bar{Z}_i+\bar{Z}_iZ_i)\psiup-X_0\psiup
\right)
 +s^2{\sum}_{i=1}^r|Z_i\psiup|^2\\
 =s\,dd^c\psiup(\tauup)+s^2{\sum}_{i=1}^r|Z_i\psiup|^2.
\end{align*}
Thus the condition of Proposition~\ref{p8.7} is satisfied for a suitable $\tauup\in\Kf$ near $p_0$ 
either
when $\partial\Omega$ is $\Theta$-non-characteristic at $p_0$, 
by taking $s\gg{1}$, or, in  case
$\partial\Omega$ is $\Theta(\tauup)$-characteristic at $p_0$, if 
$dd^c\psiup(\tauup)(p_0)>0$.
\end{proof}
\begin{rmk}
 We observe that strict $1$-convexity at $p_0$ 
 implies that $\partial\Omega$ is $\tilde{\Theta}$-non-characteristic at $p_0$.
\end{rmk}
\section{Cauchy problem for \textsl{CR} functions - Existence} \label{s9}
In this section we will investigate properties of $CR$ functions on 
$CR$ manifolds satisfying weak $2$-pseudo-concavity assumptions.
\begin{prop} \label{pr8.1}
Let $\Omega$ be an open subset of a $CR$ manifold $M$ enjoying
property  $\Psi^{we}(2)$. Assume that 
$p_0$ is a smooth, strongly-$1$-convex,
  $\Theta$-non-characteristic point of $\partial\Omega$.  Then, for every 
  relatively compact open neighborhood $U$ 
 of $p_0$ in $M$, we can find an open neighborhood $U'$ of $p_0$ in $U$ such that 
 \begin{equation}
 |f(p)|\leq\sup_{U\cap\partial\Omega}|f|,\;\;\forall p\in{U}'\cap\Omega,\;\;\forall 
 \;f\in\Ot_{\!{M}}(\Omega)\cap
 \Co^2(\bar{\Omega}),
 \end{equation}
 and strict inequality holds if $f$ is not a constant on $U'\cap\Omega$.
\end{prop} 
\begin{proof}
 We can assume that $\Omega$ is locally defined near $p_0$ by a  real valued 
 ${\rhoup\in\Ci(U)}$: 
\begin{equation*}
 U\cap\Omega=\{p\in{U}\mid \rhoup(p)<0\},\quad
 \text{and}\;\; \exists\;{Z}\in\Theta(U)\;
 \; \text{s.t.}\;\; (Z\rhoup)(p_0)\neq{0}.
\end{equation*}
To make local bumps of $\partial\Omega$ near $p_0$, we fix smooth coordinates
$x$ centered at $p_0$, that we can take for simplicity defined on $U,$ and, for a 
nonnegative
real valued smooth
function $\chiup(t)\in\Cic(\R)$,  equal to $1$ on
a neighborhood of $0$, set $\phiup_{\epsilon}(p)=e^{-1/\epsilon}\chiup(|x|/\epsilon)$.
Then we consider the domains 
\begin{equation*}
 U_{\epsilon}^-=\{p\in{U}\mid - \phiup_{\epsilon}(p) <\rhoup(p)<0\}.
\end{equation*}
There is $\epsilon_0>0$ such that $U_{\epsilon}^-\Subset{U}$ 
and 
the points of
${N_{\epsilon}''=\partial{U}^-_\epsilon\cap\Omega}$ 
are smooth and
$\Theta$-non-characteristic for all $0<\epsilon\leq\epsilon_0.$ 
In fact $N_{\epsilon}''$ is a small deformation of 
$N_{\epsilon}'=\{\phiup_{\epsilon}>0\}\cap\partial\Omega$, which is smooth and
$\Theta$-non-characteristic for $0<\epsilon\ll{1}$.
\par
We claim that, for  sufficiently small $\epsilon>0,$  the modulus
$|f|$ of
any 
function
$f\in\Ot_{\!{M}}(U^-_\epsilon)\cap\Co^2(\bar{U}^-_\epsilon)$ 
attains its maximum on $N$. We argue by contradiction. \par 
If our claim is false, then for all $0<\epsilon\leq\epsilon_0$ we can find 
$p_{\epsilon}\in{N}_{\epsilon}''$ and 
${f_{\epsilon}\in\Ot_{\!{M}}(U^-_\epsilon)\cap\Co^2(\bar{U}_\epsilon^-)}$ with 
$|f(p_{\epsilon})|>|f(p)|$ for all $p\in{U}_{\epsilon}^-$. 
In fact $\Psi^{we}(2)$ implies the maximum modulus principle and therefore
the maximum of $|f_{\epsilon}|$ is attained on the boundary of $U^-_\epsilon.$
By Proposition~\ref{lm5.2}, 
this implies that there is $\xiup_{\epsilon}\in{H}^0_{M,p_{\epsilon}}(U^-_{\epsilon})$ 
such that $\Lf^{N''_{\epsilon}}_{\xiup_\epsilon}\geq{0}$. 
By the strong-$1$-convexity assumption, there is 
$\tauup_0\in\Kf_{d\rhoup^\perp,p_0}$ 
(see Notation~\ref{ntz7.1})
such that ${dd^c\rhoup(\tauup_0)>0}$.
For $\epsilon_\nu\searrow{0}$, the sequence $\{p_{\epsilon_\nu}\}$ converges to
$p_0$. We can take a function $\tilde{\rhoup}\in\Ci(U)$ 
such that $\tilde{\rhoup}$ agrees to the second order
with $(\rhoup+\phiup_{\epsilon_\nu})$ at $p_{\epsilon_\nu}$,
for all $\nu$, and  with $\rho$ 
at~$p_0$.
\par 
We obtain a contradiction, because $\tauup_0$
belongs to $\Kf_{d\tilde{\rhoup}^\perp,p_0}=\Kf_{d\rhoup^\perp,p_0}$
 and therefore,
by $\Psi^{we}(2)$, is a cluster point of a sequence of 
elements $\tauup_{\epsilon_\nu}\in\Kf_{d\tilde\rhoup^\perp,p_{\epsilon_\nu}}=
\Kf_{d(\rhoup+\phiup_{\epsilon_\nu})^\perp,p_{\epsilon_\nu}}$, 
and $dd^c\tilde{\rhoup}(\tauup_{\epsilon_\nu})=
dd^c(\rhoup+\phiup_{\epsilon_\nu})(\tauup_{\epsilon_\nu})\leq{0}$ by Proposition~\ref{lm5.2} and Corollary~\ref{cor4.8}.
In fact, $dd^c(\rhoup+\phiup_{\epsilon_{\nu'}})(\tauup_{\epsilon_{\nu'}})
\longrightarrow dd^c\rhoup(p_0)(\tauup_0)$ when $\tauup_{\epsilonup_{\nu'}}
\longrightarrow\tauup_0$.
\end{proof} 
\begin{thm}
Let $\Omega$ be an open subset of a $CR$ manifold $M$ enjoying
property  $\Psi^{we}(2)$ and $N$ a relatively open subset of $\partial\Omega$, consisting
of smooth, strongly-$1$-convex, $\Theta$-non-characteristic points. If $M$ is locally
$CR$-embeddable at all points of $N$, then we can find an open neighborhood $U$ of $N$ in $M$
such that for every $f_0\in\Ot_{\!{N}}(N)$ there is a unique 
$f\in\Ot_{\!{M}}(U\cap\Omega)\cap\Ci(\overline{ U\cap\Omega})$ 
with $f=f_0$ on $N$.
\end{thm} 
\begin{proof}
The result easily follows from the approximation theorem in \cite{ba-tr} and the estimate of
Proposition~\ref{pr8.1}
\end{proof}

\providecommand{\bysame}{\leavevmode\hbox to3em{\hrulefill}\thinspace}
\providecommand{\MR}{\relax\ifhmode\unskip\space\fi MR }
\providecommand{\MRhref}[2]{%
  \href{http://www.ams.org/mathscinet-getitem?mr=#1}{#2}
}
\providecommand{\href}[2]{#2}


\end{document}